\documentclass[11pt,oneside,a4paper]{article} 

\usepackage[english]{babel}	
\usepackage[T1]{fontenc}		
\usepackage[utf8]{inputenc}	
\usepackage{lmodern}	
\usepackage{amsmath}
\usepackage{amsthm}	
\usepackage{amssymb}
\usepackage{mathrsfs}
\usepackage{paralist}	
\usepackage{geometry}			
\usepackage[all]{xy}
\usepackage[colorlinks = true,
            linktocpage = true,
            linkcolor = blue,
            urlcolor  = blue,
            citecolor = blue,
            anchorcolor = black]{hyperref}
\usepackage{xcolor}
\usepackage{csquotes}

\numberwithin{equation}{section} 

\newtheoremstyle{bfnote}
	{}{}
	{\itshape}{}
	{\bfseries}{.}
	{ }
	{\thmname{#1}\thmnumber{ #2}\thmnote{ (#3)}}

\newtheoremstyle{itremark}%
	{}{}%
	{}{}%
	{\itshape}{.}%
	{ }%
	{\thmname{#1}\thmnumber{ #2}\thmnote{ (#3)}}

\newtheoremstyle{itcase}%
	{}{}%
	{}{}%
	{\itshape}{:}%
	{ }%
	{\thmname{#1}\thmnote{ #3}}
	
\theoremstyle{bfnote}
\newtheorem{thm}{Theorem}[section]
\newtheorem{lem}[thm]{Lemma}
\newtheorem{prop}[thm]{Proposition}
\newtheorem{cor}[thm]{Corollary}

\theoremstyle{definition}
\newtheorem{Def}[thm]{Definition}

\theoremstyle{itremark}
\newtheorem{rem}[thm]{Remark}

\theoremstyle{itcase}

\newcommand{\Z}{\mathbb{Z}}

\newcommand{\R}{\mathbb{R}}
\newcommand{\C}{\mathbb{C}}
\newcommand{\F}{\mathscr{F}}
\newcommand{\G}{\mathscr{G}}
\newcommand{\M}{\mathscr{M}}
\newcommand{\D}{\mathscr{D}}

\newcommand{\flag}{\mathbf{H}}
\newcommand{\Dom}{\mathfrak{D}}

\DeclareMathOperator{\End}{End}
\DeclareMathOperator{\tr}{tr}

\DeclareMathOperator{\Span}{span}
\DeclareMathOperator{\Hom}{Hom}

\DeclareMathOperator{\Id}{Id}

\DeclareMathOperator{\Vol}{Vol}

\DeclareMathOperator{\codim}{codim}

\date{}
\author{\textsc{Thibault Langlais}\footnote{Mathematical Institute, University of Oxford, Oxford OX2 6GG, United Kingdom. \newline E-mail address: langlais@maths.ox.ac.uk. ORCID iD: \href{https://orcid.org/0000-0002-6434-2988}{0000-0002-6434-2988}.}}
\title{\bfseries{Geometry and periods of $G_2$-moduli spaces}}

\begin{document}


\maketitle

\begin{abstract}
    \noindent{}This paper is concerned with the geometry of the moduli space $\mathscr{M}$ of torsion-free $G_2$-structures on a compact $G_2$-manifold $M$, equipped with the volume-normalised $L^2$-metric $\mathscr{G}$. When $b^1(M) = 0$, this metric is known to be of Hessian type and to admit a global potential. Here we give a new description of the geometry of $\mathscr{M}$, based on the observation that there is a natural way to immerse the moduli space into a homogeneous space $\mathfrak{D}$ diffeomorphic to $GL(n+1)/ (\{\pm 1\} \times O(n))$, where $n = b^3(M) - 1$. We point out that the formal properties of this immersion $\Phi : \mathscr{M} \rightarrow \mathfrak{D}$ are very similar to those of the period map defined on the moduli spaces of Calabi--Yau threefolds. With a view to understand the curvatures of $\mathscr{G}$, we also derive a new formula for the fourth derivative of the potential and relate it to the second fundamental form of $\Phi(\mathscr{M}) \subset \mathfrak{D}$.
\end{abstract}

\small
\tableofcontents

\normalsize 



    \section{Introduction and motivation} 

$G_2$-manifolds are an exceptional class of Riemannian $7$-manifolds with remarkable geometric features; in particular, they are automatically Ricci-flat and admit non-trivial parallel spinors. For these reasons, $G_2$-manifolds are important in quantum gravity theories, especially in M-theory where they play the same role as Calabi--Yau manifolds in string theory. Both in the $G_2$ and Calabi--Yau cases, an important problem in mathematics and physics is to describe the geometry of the moduli spaces.

For Calabi--Yau manifolds, there is an extensive literature in complex geometry describing the properties of the moduli spaces, and one can consider separately the deformations of the K\"ahler class and of the complex structure. For the first case, the K\"ahler classes of a compact Calabi--Yau manifold $Y$ form an open convex cone inside $H^{1,1}(Y)$. It admits a natural Riemannian metric which turns out to be the Hessian of the potential $-\log \Vol$, where the volume of a K\"ahler class is (up to a combinatorial factor) just its top exterior power evaluated on the fundamental class of $Y$. Therefore all the geometric invariants of this metric are determined by the derivatives of the potential, which can in turn be expressed in terms of the intersection form on $H^{\bullet}(Y)$, providing a link between the topology of $Y$ and the geometry of its K\"ahler cone \cite{huybrechts2001products,wilson2004sectional,trenner2011asymptotic}. 

On the other hand, for the moduli space of complex structures of (polarised) Calabi--Yau manifolds the relevant metric to consider is the Weil--Petersson metric. It was shown by Tian \cite{tian1987smoothness} and Todorov \cite{todorov1989weil} that this metric is determined by the \emph{period map}, which was first introduced by Griffiths \cite{griffiths1968periodsi,griffiths1968periodsii}. Using results of Schmid on the asymptotic behaviour of the period map \cite{schmid1973variation} and Viehweg on the quasiprojectivity of the moduli spaces \cite{viehweg1991quasi}, the relation between the Weil--Petersson metric and the period map was axiomatised by Lu and Sun \cite{lu2001hodge,lu2004weil}, who notably deduced the finiteness and rationality of the volume of the moduli spaces \cite{lu2006weil}.

By contrast, much less is known about the geometry of $G_2$-moduli spaces. Joyce proved that the moduli space of torsion-free $G_2$-structures on a compact $7$-manifold $M$ (when nonempty) is a smooth manifold of dimension $b^3(M)$, locally modelled on an open cone in $H^3(M)$ \cite{joyce1996compacti}; thus the moduli space is even an \emph{affine manifold}. Moreover, it is naturally immersed as a Lagrangian submanifold of $H^3(M) \oplus H^4(M)$ \cite{joyce1996compactii}. However, all of these results are local, and little is known about the global structure of the moduli space, partly because of the lack of an analogue of Yau's theorem \cite{yau1978ricci} in $G_2$-geometry. Nevertheless, there have been some recent advances on the topology of the moduli spaces: it is known that the quotient of the space of torsion-free $G_2$-structures by the full space of diffeomorphisms is disconnected in some cases \cite{crowley2015analytic}; moreover, the quotient of the space of torsion-free $G_2$-structures by the group of diffeomorphisms isotopic to the identity may be non-aspherical \cite{crowley2025path}.

From a geometric perspective, Hitchin first noticed that the Hessian of the volume functional is nondegenerate \cite{hitchin2000geometry}, and when $b^1(M) = 0$ it defines a metric with Lorentzian signature on the moduli space. Around the same time, it was pointed out in the physics literature that the Hessian of the potential $-3\log \Vol$ is positive definite, and coincides with the volume-normalised $L^2$-metric \cite{gukov2000solitons,beasley2002note,gutowski2001moduli,house2005m}. Unlike K\"ahler cones however, the volume is not a merely polynomial function of the cohomology class of the $G_2$-structure, and the high degree of nonlinearity of this function makes the geometry of $G_2$-moduli spaces very difficult to understand. Grigorian and Yau \cite{grigorian2009local} obtained formulas for the derivatives of the potential up to order $4$, but their expressions are difficult to interpret geometrically. Nevertheless, an interesting feature of these formulas is their similarity with the equations describing the geometry of the moduli spaces of complex structures on Calabi--Yau threefolds. Further similarities were exhibited by the work of Karigiannis and Leung \cite{karigiannis2009hodge}, who developed a notion of Intermediate Jacobians for $G_2$-manifolds. There is also ongoing work by Karigiannis and Loftin \cite{karigiannis2025octonionic} about the curvatures of the moduli spaces, motivated by the conjectured existence of universal bounds for the sectional curvatures of the K\"ahler cone of Calabi--Yau manifolds \cite{wilson2004sectional}. Recently, the author obtained sufficient conditions for the limit of a one-parameter family of degenerating $G_2$-manifolds to be at finite distance in the moduli space and proved that $G_2$-moduli spaces are not always complete \cite{langlais2024incompleteness}. 

In the present paper, we give a new description of the geometry of $G_2$-moduli spaces and reinterpret the similarities with the Calabi--Yau case. In the future, we hope that this perspective could be used to gain further insights about $G_2$-moduli spaces.

    \subsection*{Organisation of the paper}

Let us give a brief overview of the results and the organisation of the paper. First, we gather some background about $G_2$-geometry in Section \ref{section:background}. In Section \ref{section:derivatives}, we derive a new formula for the fourth derivative of the potential which only depends on the lower order derivatives and some extra terms related to the variations of the space of harmonic forms. We prove that when $M = T^7 / \Gamma$ or $M = (T^3 \times K3) / \Gamma$ these extra terms vanish, and that the resulting equation for the potential implies that the moduli spaces are locally symmetric. Beyond these cases, the extra terms cannot be computed explicitly, impeding a further understanding of the geometry of the moduli space using computations in local coordinates. Motivated by this difficulty, we introduce a new perspective in Section \ref{section:periods}. We observe that the variation of the Hodge decomposition of $H^3(M) \oplus H^4(M)$ defines an immersion of the moduli space into a homogeneous space $\Dom$ diffeomorphic to $GL(n+1) / (\{ \pm 1 \} \times O(n))$, where $n = b^3(M) - 1$. We show that the properties of this map are very analogous to the period map of Calabi--Yau threefolds, and that it determines the geometry of the moduli space in a natural way. Finally, we relate this to the results of the previous section, by proving that the extra terms in the formula of the fourth derivative of the potential are intrinsically related to the second fundamental form of the moduli space seen as an immersed submanifold of $\Dom$.

    \subsection*{Acknowledgements}

I would like to thank Spiro Karigiannis for interesting discussions related to the material presented in this article, and my supervisor Jason D. Lotay for his advice and support. I am also grateful to an anonymous referee for their valuable comments, which helped me greatly improve the exposition of the present article.

This research is supported by scholarships from the Clarendon Fund and the Saven European Programme. Part of this material is also based upon work supported by the National Science Foundation under Grant No. DMS-1928930, while the author was in residence at the Simons Laufer Mathematical Sciences Institute (formerly MSRI) in Berkeley, California, during the Fall 2024 semester.



    \section{Background on $G_2$-geometry}      \label{section:background}

This section gathers some basic notions of $G_2$-geometry. In \S\ref{subsection:posform}, we recall the definition of positive forms in $\R^7$ and a few elements of their linear algebra. $G_2$-manifolds are introduced in \S\ref{subsection:g2manifolds}, and their moduli spaces in \S\ref{subsection:modulispaces}.

    \subsection{Positive forms on $\R^7$}       \label{subsection:posform}

Let us consider $\R^7$ equipped with its standard orientation and denote by $\R^*_7$ its dual space. A $3$-from $\varphi \in \Lambda^3 \R^*_7$ is said to be \emph{positive} if for any $v \in \R^7 \backslash \{0\}$ we have
\begin{equation}
    (v \lrcorner \varphi) \wedge (v \lrcorner \varphi) \wedge \varphi > 0
\end{equation}
relative to the standard orientation. Here $\cdot \lrcorner \cdot$ denotes the interior product of a vector in $\R^7$ and an alternating form in $\Lambda (\R^*_7)$. The set $\Lambda^3_+ \R^*_7$ of positive forms is nonempty and open in $\Lambda^3 \R^*_7$, and is acted upon transitively by the group of orientation-preserving automorphisms $GL_+(7)$. The stabiliser of any positive form is conjugate to the group $G_2 \subset SO(7)$. This is a compact, simple Lie group of dimension $14$. A positive form $\varphi \in \Lambda^3_+ \R^*_7$ canonically determines an inner product on $\R^7$, which we denote by $g_\varphi$ or $\langle \cdot, \cdot \rangle_\varphi$, and a $7$-from $\mu_\varphi \in \Lambda^7 \R^*_7$ characterised by
\begin{align}
    \begin{split}
        (v \lrcorner \varphi) \wedge (u \lrcorner \varphi) \wedge \varphi & = 6 \langle u, v \rangle_\varphi \mu_\varphi, ~~ \forall u,v \in \R^7, ~~ \text{and} \\
        |\varphi|_{g_\varphi}^2 & = 7 .
    \end{split}
\end{align}
The dual $4$-form of $\varphi$ with respect to the Hodge operator $*_\varphi$ associated with $g_\varphi$ is commonly denoted by $\Theta(\varphi) \in \Lambda^4 \R^*_7$. The maps $\varphi \mapsto g_\varphi$, $\varphi \mapsto \mu_\varphi$, $\varphi \mapsto *_\varphi$ and $\varphi \mapsto \Theta(\varphi)$ are non-linear and equivariant under the action of $GL_+(7)$.

Let us fix a positive form $\varphi$ on $\R^7$, and identify the stabiliser of $\varphi$ with $G_2$. The exterior algebra $\Lambda \R^*_7$ can be decomposed into irreducible representations of $G_2$ as follows. The representation $\R^*_7$ is irreducible, as $G_2$ acts transitively on the unit sphere. The space of $2$-forms can be decomposed as:
\begin{equation*}
    \Lambda^2 \R^*_7 = \Lambda^2_{14} \oplus \Lambda^2_7
\end{equation*}
where $\Lambda^2_{14}$ is isomorphic to the Lie algebra of $G_2$ and $\Lambda^2_7 \simeq \R^*_7$. In particular, any $\omega \in \Lambda^2 \R^*_7$ can be written uniquely as
\begin{equation*}
    \omega = v \lrcorner \varphi + \chi, ~~~ v \in \R^7, ~ \chi \in \Lambda^2_{14} .
\end{equation*}
In order to decompose $\Lambda^3 \R^*_7$, let us introduce a bilinear map $\End(\R^7) \otimes \Lambda(\R^*_7) \rightarrow \Lambda(\R^*_7)$ defined by:
\begin{equation}
    h \cdot \eta = \left. \frac{d}{dt} \right|_{t=0} (e^{th})^* \eta = \eta(h \cdot, \cdot, \cdots) + \cdots + \eta(\cdots, \cdot, h\cdot), ~~~ \forall (h,\eta) \in \End(\R^7) \times \Lambda(\R^*_7) .
\end{equation}
Up to a sign, this is the derivative of the action of $GL(7)$ on $\Lambda \R^*_7$. Since $GL_+(7)$ acts transitively on $\Lambda^3_+ \R^*_7$ which is open in $\Lambda^3 \R^*_7$, the map $h \in \End(\R^7) \mapsto h \cdot \varphi \in \Lambda^3 \R^*_7$ is onto. The representation $\End(\R^7)$ can be decomposed as:
\begin{equation*}
    \End(\R^7) \simeq \Lambda^2 \R^*_7 \oplus S^2 \R^*_7 \simeq \Lambda^2_{14} \oplus \Lambda^2_7 \oplus \R \oplus S^2_0 \R^*_7 
\end{equation*}
where $\Lambda^2_{14}$ is identified with the Lie algebra of $G_2$ and $S^2_0 \R^*_7$ is isomorphic to the space of trace-free self-adjoint endomorphisms with respect to $g_\varphi$. The kernel of the above map $\End(\R^7) \rightarrow \Lambda^3 \R^*_7$ is $\Lambda^2_{14}$, and therefore we obtain the decomposition:
\begin{equation*}
    \Lambda^3 \R^*_7 = \Lambda^3_1 \oplus \Lambda^3_7 \oplus \Lambda^3_{27}
\end{equation*}
where $\Lambda^3_7 \simeq \R^*_7$ and $\Lambda^3_{27} \simeq S^2_0 \R^*_7$. In particular, any $3$-form $\eta \in \Lambda^3 \R^*_7$ can be written uniquely as
\begin{equation*}
    \eta = \lambda \varphi + v \lrcorner \Theta(\varphi) + \nu, ~~~ \lambda \in \R, ~ v \in \R^7, ~ \nu \in \Lambda^3_{27} .
\end{equation*}
As $\Lambda^k \R^*_7 \simeq \Lambda^{7-k} \R^*_7$ under Hodge duality, this give a full decomposition of $\Lambda \R^*_7$. We denote by $\pi_m$ the projection of $\Lambda^k \R^*_7$ onto $\Lambda^k_m$.

We finish these generalities with a few useful formulas for the first variation of various tensors associated with an inner product or a positive form on $\R^7$, and some interesting consequences. First, we begin with some properties of the bilinear map $\End(\R^7) \otimes \Lambda(\R^*_7) \rightarrow \Lambda(\R^*_7)$ previously defined:

\begin{lem}     \label{lem:derivation}
    For $h \in \End(\R^7)$ we denote by $\delta_h : \Lambda(\R^*_7) \rightarrow \Lambda(\R^*_7)$ the linear map $\eta \mapsto h \cdot \eta$. Then for $h,h^\prime \in \End(\R^7)$ the following properties are satisfied:

    \begin{enumerate}[(i)]
        \item The map $\delta_h$ is a derivation of degree $0$ of $\Lambda(\R^*_7)$. That is, it preserves the degree of forms and $h \cdot (\omega \wedge \omega^\prime) = (h \cdot \omega) \wedge \omega^\prime + \omega \wedge (h \cdot \omega^\prime)$ for any $\omega, \omega^\prime \in \Lambda(\R^*_7)$.

        \item $[\delta_h, \delta_{h^\prime}] = - \delta_{[h,h^\prime]}$. 
        
        \item If $h$ is (anti-)self-adjoint for some inner product on $\R^7$, then $\delta_h$ is (anti-)self-adjoint for the induced inner product on $\Lambda (\R^*_7)$.
    \end{enumerate}
\end{lem}

\begin{proof}
    That $\delta_h$ is a derivation of degree $0$ can be seen by differentiating the identity $(e^{th})^* (\omega \wedge \omega^\prime) = (e^{th})^* \omega \wedge (e^{th})^* \omega^\prime$. Moreover, by definition $\delta : \End(\R^7) \rightarrow \End(\Lambda(\R^*_7))$ is the negative of the natural action of the Lie algebra $\End(\R^7)$ on $\Lambda(\R^*_7)$, and thus $[\delta_h, \delta_{h^\prime}] = - \delta_{[h,h^\prime]}$. Last, if $h$ is $g$-self-adjoint and $\omega \in \R^*_7$, then the dual vector of $\delta_h \omega = \omega \circ h$ is $h(v)$, where $v \in \R^7$ is dual to $\omega$. From this it follows that $\delta_h$ is self-adjoint for the inner product induced by $g$ on $\R^*_7$, and thus on $\Lambda(\R^*_7)$. We can argue similarly for the case where $h$ is anti-self-adjoint for $g$, since then the dual vector of $\omega \circ h \in \R^*_7$ is $-h(v)$ if $v \in \R^7$ is the vector dual to $\omega$.
\end{proof}

The next lemma gathers a few useful identities which are easy to check.

\begin{lem}    \label{lem:firstvariationh}
    Let $g$ be an inner product on $\R^7$ and $h \in \End(\R^7)$, and consider a $1$-parameter family of inner products $g_t$ such that $g_0 = g$ and $\left. \frac{d g_t}{dt} \right|_{t=0} = g(h \cdot, \cdot) + g(\cdot, h \cdot)$. Let $\omega, \omega^\prime \in \Lambda^k \R^*_7$ for some $0 \leq k \leq 7$. Then we have the following first variation formulas:
    \begin{align*}
        \left. \frac{d}{dt} \right|_{t=0} \langle \omega, \omega^\prime \rangle_{g_t}  & = - \langle h \cdot \omega, \omega^\prime \rangle_g - \langle \omega, h \cdot \omega^\prime \rangle_g , \\
        \left. \frac{d}{dt} \right|_{t=0} *_{g_t} \omega & = h \cdot (*_g \omega) - *_g (h \cdot \omega) , \\
        \left. \frac{d}{dt} \right|_{t=0} \mu_{g_t} & = \tr(h) \mu_g .
    \end{align*}
\end{lem}

These two lemmas have a few consequences that will be useful in the rest of the article. First note that if if $h$ is self-adjoint for $g$, then $\delta_h$ is self-adjoint for the inner product induced by $g$ on $\Lambda(\R^*_7)$ and thus with the notations of the above lemma we have
\begin{equation*}
    \left. \frac{d}{dt} \right|_{t=0} \langle \omega, \omega^\prime \rangle_{g_t} = - 2 \langle h \cdot \omega, \omega^\prime \rangle_g
\end{equation*}
for any $\omega, \omega^\prime \in \Lambda^k \R^*_7$. This implies:

\begin{cor}         \label{cor:hstar}
    Let $h, h^\prime \in \End(\R^7)$, and suppose that $h$ is a trace-free endomorphism, self-adjoint with respect to an inner product $g$, and $h^\prime$ is anti-self-adjoint for $g$. Then for any $\omega \in \Lambda^k \R^*_7$ we have:
    \begin{equation*}
        h \cdot (*_g \omega) = - *_g (h \cdot \omega), ~~~ \text{and} ~~~ h^\prime \cdot (*_g \omega) = *_g(h^\prime \cdot \omega) .
    \end{equation*}
\end{cor}

\begin{proof}
    Consider the family of inner products $g_t = (e^{th})^* g$. Using the previous lemmas, we can differentiate the identity $\omega^\prime \wedge *_{g_t} \omega = \langle \omega^\prime, \omega \rangle_{g_t} \mu_{g_t}$ at $t=0$ which yields:
    \begin{equation*}
        \omega^\prime \wedge h \cdot (*_g \omega) - \omega^\prime \wedge *_g (h \cdot \omega) = - 2 \langle \omega^\prime, h \cdot \omega \rangle_g \mu_g = - 2 \omega^\prime \wedge *_g (h \cdot \omega)
    \end{equation*}
    and hence $\omega^\prime \wedge h \cdot (*_g \omega) = - \omega^\prime \wedge *_g (h \cdot \omega)$ for any $\omega^\prime \in \Lambda^k \R^*_7$, which proves the first identity.

    For the second identity, we note that since $h^\prime$ is anti-self-adjoint for $g$, the linear isomorphisms $e^{th^\prime}$ preserve $g$, and thus
    \begin{equation*}
        *_g \omega  = e^{th^\prime} (*_g e^{-th^\prime} \omega) 
    \end{equation*}
    for any $t$, and differentiating at $t=0$ it follows that $ h^\prime \cdot (*_g \omega) - *_g(h^\prime \cdot \omega) = 0$.
\end{proof}

Another useful consequence to note is:

\begin{cor}     \label{cor:symyukawa}
    If $\varphi$ is a positive form on $\R^7$, then the cubic form
    \begin{equation*}
        (h_1,h_2,h_3) \in S^2 \R^*_7 \times S^2 \R^*_7 \times S^2 \R^*_7 \longmapsto \langle h_3 \cdot h_1 \cdot \varphi, h_2 \cdot \varphi \rangle_\varphi \in \R
    \end{equation*}
    is fully symmetric.
\end{cor}

\begin{proof}
    The identity
    \begin{equation*}
        \langle h_3 \cdot h_1 \cdot \varphi , h_2 \cdot \varphi \rangle_\varphi = \langle h_1 \cdot \varphi, h_3 \cdot h_2 \cdot \varphi \rangle_\varphi = \langle h_3 \cdot h_2 \cdot \varphi, h_1 \cdot \varphi \rangle_\varphi
    \end{equation*}
    holds because $\delta_{h_3}$ is self-adjoint for the inner product induced by $\varphi$ in $\Lambda(\R^*_7)$. Thus the cubic form is symmetric under permutation of $h_1$ and $h_2$. To prove that it is also symmetric under permutation of $h_1$ and $h_3$, note that since $[\delta_{h_3},\delta_{h_1}] = - \delta_{[h_3,h_1]}$ we have
    \begin{equation*}
         \langle h_3 \cdot h_1 \cdot \varphi, h_2 \cdot \varphi \rangle_\varphi - \langle h_1 \cdot h_3 \cdot \varphi, h_2 \cdot \varphi \rangle_\varphi = \langle [h_1,h_3] \cdot \varphi , h_2 \cdot  \varphi \rangle_\varphi = 0
    \end{equation*}
    where the last equality follows from the fact that $[h_1,h_3]$ is anti-self-adjoint, and thus $[h_1,h_3] \cdot \varphi \in \Lambda^3_7$ is orthogonal to $h_2 \cdot \varphi \in \Lambda^3_1 \oplus \Lambda^3_{27}$.
\end{proof}

Finally, we record the following well-known first variations formulas:

\begin{lem}            \label{lem:firstvariationphi}
    Let $\varphi$ be a positive form on $\R^7$, $\eta \in \Lambda^3 \R^*_7$ and let $h \in \End(\R^7)$ be the unique endomorphism orthogonal to $\Lambda^2_{14}$ such that $h \cdot \varphi = \eta$. Let $\varphi_t$ be a $1$-parameter family of positive forms in $\R^7$ such that $\varphi_0 = \varphi$ and $\left. \frac{d \varphi_t}{dt} \right|_{t=0} = \eta$, and let $\omega, \omega^\prime \in \Lambda^k \R^*_7$ for some $0 \leq k \leq 7$. Then we have the following first variation formulas:
    \begin{align*}
        \left. \frac{d}{dt} \right|_{t=0}  \langle\omega, \omega^\prime \rangle_{\varphi_t} & = - \langle h \cdot \omega, \omega^\prime \rangle_{\varphi} - \langle \omega, h \cdot \omega^\prime \rangle_{\varphi}, \\
        \left. \frac{d}{dt} \right|_{t=0} *_{\varphi_t} \omega & = h \cdot (*_\varphi \omega) - *_{\varphi} (h \cdot \omega) , \\
        \left. \frac{d}{dt} \right|_{t=0} \mu_{\varphi_t} & = \tr(h) \mu_\varphi, \\
        \left. \frac{d}{dt} \right|_{t=0} \Theta(\varphi_t) & = \frac{4}{3} *_\varphi \pi_1(\eta) + *_\varphi \pi_7(\eta) - *_\varphi \pi_{27}(\eta) \cdot
    \end{align*}
\end{lem}

    \subsection{$G_2$-manifolds}        \label{subsection:g2manifolds}

Let $M^7$ be an oriented $7$-dimensional manifold. A \emph{$G_2$-structure} corresponds to the data of a $3$-form $\varphi$ such that $\varphi_p \in T_pM$ is positive for every $p \in M$. The properties of positive forms on $\R^7$ carry over to $G_2$-structures on manifolds; in particular a $G_2$-structure $\varphi \in \Omega^3(M)$ determines a Riemannian metric $g_\varphi$, a volume form $\mu_\varphi$ and a $4$-form $\Theta(\varphi) = *_\varphi \varphi$. It is called \emph{torsion-free} if $\varphi$ is parallel for the Levi-Civita connection of $g_\varphi$, or equivalently if $\varphi$ is closed and co-closed, that is, $d\varphi = 0 = d(\Theta(\varphi))$ \cite{fernandez1982riemannian}. If this is satisfied, then the holonomy group of $g_\varphi$ is isomorphic to a subgroup of $G_2$, and in particular the metric $g_\varphi$ is \emph{Ricci-flat} \cite{bonan1966sur}. The existence of metrics with full holonomy $G_2$ was first proved by Bryant \cite{bryant1987metrics} for local metrics, Bryant--Salamon for complete ones \cite{bryant1989construction}, and Joyce \cite{joyce1996compacti} on compact manifolds. A manifold $M$ endowed with a torsion-free $G_2$-structures $\varphi$ is called a \emph{$G_2$-manifold}. In this part we give some background on the geometry of such manifolds.

If $M$ is equipped with a $G_2$-structure $\varphi$ (not necessarily torsion-free), there is an associated splitting of the exterior bundle $\Lambda(T^*M)$ and identifications
\begin{align*}
    T^*M & \simeq TM, \\
    \Lambda^2 T^*M & = \Lambda^2_7 T^*M \oplus \Lambda^2_{14} T^*M, ~~~~~~~~~~~~~~~~~~ \Lambda^2_7 T^*M \simeq T^*M, \\
    \Lambda^3 T^*M & = \Lambda^3_1 T^*M \oplus \Lambda^3_7 T^*M \oplus \Lambda^3_{27} T^*M, ~~~~ \Lambda^3_1 T^*M \simeq \underline \R, ~ \Lambda^3_7 T^*M \simeq T^*M, \\
    \Lambda^k T^*M & \simeq \Lambda^{7-k} T^*M, ~~~~~~~~~~~~~~~~~~~~~~~~~~~~~~ k = 0, \ldots , 7
\end{align*}
where $\underline \R$ is the trivial real line bundle $M \times \R$. There is a corresponding splitting of the algebra of differential forms on $M$, and we will denote $\Omega^k(M) = \oplus_m \Omega^k_m(M)$ where $\Omega^k_m(M) = C^\infty(\Lambda^k_m T^*M)$, and $\pi_m$ the projection of $\Omega^k(M)$ onto $\Omega^k_m(M)$. In particular, any $2$-form $\omega$ on $M$ can be written uniquely as 
\begin{equation*}
    \omega = \xi \lrcorner \varphi + \chi, ~~~ \xi \in C^\infty(TM), ~ \chi \in \Omega^2_{14}(M),
\end{equation*}
and any $3$-form $\eta$ can be written uniquely as
\begin{equation*}
    \eta = f \varphi + \xi \lrcorner \Theta(\varphi) + \nu, ~~~ f \in C^\infty(M), ~ \xi \in C^\infty(TM), ~ \nu \in \Omega^3_{27}(M) .
\end{equation*}
Another useful way to describe a $3$-form is to decompose $\End(TM)$ as
\begin{equation*}
    \End(TM) \simeq \Lambda^2 T^*M \oplus S^2 T^*M = \Lambda^2_{14} T^*M \oplus \Lambda^2_7 T^*M \oplus \R g_\varphi \oplus S^2_0 T^*M
\end{equation*}
where $S^2_0 T^*M \simeq \Lambda^3_{27} T^*M$. Then for any $3$-form $\eta \in \Omega^3(M)$, there exists a unique section $h \in C^\infty(\End(TM))$ orthogonal to $\Omega^2_{14}(M)$ such that $\eta = h \cdot \varphi$. In particular, $\pi_7(\eta) = 0$ if and only if $h$ is a self-adjoint endomorphism for the metric $g_\varphi$.

Assume now that $M$ is a compact, connected, oriented manifold endowed with a torsion-free $G_2$-structure $\varphi$. Due to a Weitzenb\"ock formula, the Laplacian operator associated with $g_\varphi$ leaves invariant each component of the splitting $\Omega^k(M) = \oplus \Omega^k_m(M)$.  Therefore, Hodge theory yields a decomposition of the de Rham cohomology groups $H^k(M) \simeq \oplus H^k_m(M)$, and moreover isomorphic representations lead to isomorphic components in cohomology. In particular:
\begin{equation*}
H^1(M) \simeq H^2_7(M) \simeq H^3_7(M) ~~ \text{and} ~~ H^3_1(M) \simeq H^0(M) \simeq \R.
\end{equation*}
We will denote by $\mathscr{H}^k(M,\varphi)$ the space of $k$-forms harmonic with respect to $g_\varphi$, and $\mathscr{H}^k_m(M,\varphi)$ the intersection of $\mathscr{H}^k(M,\varphi)$ and $\Omega^k_m(M)$. Since the metric $g_\varphi$ is Ricci-flat, $\mathscr{H}^1(M)$ is exactly the space of parallel $1$-forms on $M$, and is dual to the space of Killing fields. Moreover, the Cheeger--Gromoll splitting theorem \cite{cheeger1971splitting} implies that $g_\varphi$ has full holonomy $G_2$ if and only if the fundamental group $\pi_1(M)$ is finite \cite[Prop. 10.2.2]{joyce2000compact}. 

A weaker condition is requiring $b^1(M) = 0$. Geometrically, this prevents the existence of parallel $1$-forms and is equivalent to saying that the holonomy group of $g_\varphi$ acts on the tangent space without fixing any nonzero vector. When $b^1(M) = 0$ and $\pi_1(M)$ is infinite, $M$ has a finite cover isometric to $T^k \times N^{7-k}$ where $T^k$ is a flat torus of dimension $k$ and either one of the following possibilities holds: $k = 7$ and $N$ is reduced to a point, or $k = 3$ and $N$ is a hyperk\"ahler K3 surface, or $k = 1$ and $N$ is a Calabi--Yau threefold. In particular, the restricted holonomy group of $M$ is either one of $\{1\}$, $SU(2)$ or $SU(3)$. If we are mainly interested in metrics with full holonomy $G_2$, our results will be valid in general for $G_2$-manifolds with $b^1(M) = 0$.

    \subsection{Moduli spaces}      \label{subsection:modulispaces}

Let $M$ be a compact oriented $7$-manifold which admits torsion-free $G_2$-structures. We denote by $\D$ the group of diffeomorphisms of $M$ acting trivially on $H^3(M)$. In particular, it contains the group of diffeomorphisms isotopic to the identity, but it could be larger. The group $\D$ acts by pull-back on the space $\Omega^3_+(M)$ of $G_2$-structures on $M$, leaving invariant the subset of torsion-free $G_2$-structures. The moduli space $\M$ of torsion-free $G_2$-structures is defined as the quotient of the set of torsion-free $G_2$-structures by this action. It has a natural topology coming from the $C^\infty$-topology of $\Omega^3_+(M)$, and it was proven by Joyce \cite[Th. C]{joyce1996compacti} that it admits a compatible manifold structure of dimension $b^3(M)$, and the tangent space $T_{\varphi \D} \M$ can be identified with the space $\mathscr H^3(M,\varphi)$ of $3$-forms harmonic with respect to $g_\varphi$. Moreover, the map $\M \rightarrow H^3(M)$ sending $\varphi \D \in \mathscr M$ to the cohomology class $[\varphi ] \in H^3(M)$ is globally well-defined, and induces local diffeomorphisms between open subsets of $\M$ and open subsets of $H^3(M)$. This endows $\M$ with a natural atlas of charts with affine transition functions, and therefore $\M$ has the structure of an \emph{affine manifold}. If $(u_0,\ldots,u_n)$ is a basis of $H^3(M)$, where $n = b^3(M)-1$, we will denote by $(x^0,\ldots,x^n)$ the associated local coordinates on $\M$ and call them affine coordinates.

\begin{rem}
    In fact our definition differs slightly from the convention adopted in \cite{joyce1996compacti} (or in Joyce's monograph \cite[\S10.4]{joyce2000compact}) where one takes the quotient by the group $\D_0$ of diffeomorphisms isotopic to the identity (the resulting space may be called the ``Teichm\"uller space'' $\mathscr{T}$ of torsion-free $G_2$-structures). With our definition, we are taking a further quotient by the discrete group $\Gamma = \D / \D_0$, but since the Teichm\"uller space is locally diffeomorphic to $H^3(M)$ and $\D$ acts trivially on this space, it follows that the quotient $\M = \mathscr{T}/\Gamma$ is nonsingular and locally diffeomorphic to $H^3(M)$. Hence $\M$ is more of a ``marked moduli space'' (we fix an identification of $H^3(M)$ with $\R^{b_3(M)}$) of torsion-free $G_2$-structures on $M$, but we will just call it the moduli space for simplicity. This does not affect the results of the present article, since with either convention the moduli space is smooth and locally diffeomorphic to $H^3(M)$. 

    By means of justifying our choice of convention, let us point out the following issue. Let us look at the moduli space of torsion-free $G_2$-structures on the torus $T^7 = \R^7 / \Z^7$. It is natural to guess that it can be identified with $\Lambda^3_+ \R^*_7$. It is not difficult to see that this is true, \emph{with our convention}. The issue is that the structure of the mapping class group of $T^7$ is rather complicated, and in particular there are elements of the mapping class group which act trivially on the cohomology $H^\bullet(T^7)$ (this is true for $T^n$ for any $n \geq 5$ \cite{hatcher1978concordance,hsiang1976parametrized}). In particular, $\Gamma = \D / \D_0$ is a nontrivial group, and it acts freely transitively on the fibres of the covering map $\mathscr{T} \rightarrow \M$.
\end{rem}

The moduli space $\M$ carries a natural Riemannian metric $\G$, which can be described as follows. If $\varphi$ is a torsion-free $G_2$-structure on $M$, the tangent space $T_{\varphi \D} \M$ can be identified with the space of $3$-forms which are harmonic with respect to the metric $g_\varphi$, denoted by $\mathscr{H}^3(M,\varphi)$. Thus we can define:
\begin{equation}        \label{eq:defg}
    \G(\eta, \eta^\prime) = \frac{1}{\Vol(\varphi)} \int \langle \eta, \eta^\prime \rangle_\varphi \mu_\varphi, ~~~ \forall \eta, \eta^\prime \in \mathscr H^3(M,\varphi) \simeq T_{\varphi\D} \M,
\end{equation}
where $\Vol(\varphi)$ is the volume of $(M,g_\varphi)$, that is:
\begin{equation}
    \Vol(\varphi) = \int \mu_\varphi = \frac{1}{7} \int \varphi \wedge \Theta(\varphi) .
\end{equation}
It is perhaps worth commenting on the volume normalisation in this definition. If we denote by $\M_1 \subset \M$ the moduli space of torsion-free $G_2$-structures with unit volume, then the metric $\G$ restricts to the usual $L^2$-metric on $\M_1$, denoted by $\G_1$. Moreover, there is a diffeomorphism $\R \times \M_1 \rightarrow \M$ mapping $(t,\varphi \D)$ to $e^t \varphi \D$. It is easy to check that under this diffeomorphism $\G = 7dt^2 + \G_1$, and in particular $(\M,\G)$ splits a line and is isometric to $\R \times (\M_1,\G_1)$. Hence there is no essential difference between studying the Riemannian properties of $(\M,\G)$ and those of $(\M_1,\G_1)$. This would not be the case without the volume normalisation.

Another motivation for this choice of normalisation is that, when the first Betti number of $M$ vanishes, the metric $\G$ is Hessian. Indeed, since the volume functional is invariant under diffeomorphisms, it descends to a smooth function on the moduli space, and we can define $\mathscr F : \M \rightarrow \R$ by:
\begin{equation}
    \F(\varphi \D) = - 3 \log \Vol(\varphi) .
\end{equation}
This defines a smooth function on the moduli space, which we refer to as the \emph{potential}. If $(x^0, \ldots x^n)$ are local affine coordinates, we denote by $\F_a = \frac{\partial \F}{\partial x^a}$, $\F_{ab} = \frac{\partial^2 \F}{\partial x^a \partial x^b}$, and so on the derivatives of $\F$. If $\varphi$ is a torsion-free $G_2$-structure on $M$, we denote by $\eta_a \in \mathscr{H}^3(M,\varphi)$ the harmonic representative of the cohomology class $\frac{\partial}{\partial x^a} \in H^3(M)$. The first and second derivatives of $\F$ admit the following expressions \cite{grigorian2009local,karigiannis2009hodge}:

\begin{prop}     \label{prop:fafab}
    Let $x = (x^0, \ldots , x^n)$ be affine coordinates on $\M$ and let $\varphi$ be a torsion-free $G_2$-structure. Then the first and second derivatives of $\F$ at $\varphi \D \in \M$ read:
    \begin{equation*}
        \F_a = - \frac{1}{\Vol(\varphi)} \int \eta_a \wedge \Theta(\varphi), ~~ \text{and} ~~ \F_{ab} = \frac{1}{\Vol(\varphi)} \int \langle \eta_a, \pi_{1 \oplus 27} (\eta_b) - \pi_7(\eta_b)\rangle_\varphi \mu_\varphi .
    \end{equation*}
\end{prop}

If $b^1(M) = 0$, the harmonic $3$-forms with respect to a torsion-free $G_2$-structure have no $\Omega^3_7$-component. In this case, the second derivative of $\F$ takes the simpler form:
\begin{equation}        \label{eq:fabetaab}
    \F_{ab} = \frac{1}{\Vol(\varphi)} \int \langle \eta_a, \eta_b \rangle_\varphi \mu_\varphi .
\end{equation}
 Thus the Hessian $\F_{ab}$ is nondegenerate and positive, and in affine coordinates on $\M$
 \begin{equation*}
     \G = \G_{ab} dx^a dx^b = \F_{ab} dx^a dx^b
 \end{equation*}
where we write $dx^k dx^l$ as a short-hand for the tensor product $dx^k \otimes dx^l \in T^* \M \otimes T^*\M$, and we will use similar notations for tensor products of higher degree. Thus the metric $\G$ is the Hessian of the potential $\F$ for the flat connection induced by the map $\pi : \M \rightarrow H^3(M)$. In general, if the first Betti number of $M$ is nonzero, the Hessian of $\F$ is still nondegenerate and defines a metric of signature $(b^3(M)-b^1(M), b^1(M))$ on $\mathscr M$. Even in the case $b^1(M) = 0$, one could take the volume functional $\Vol$ instead of $\F$ as a potential, which has nondegenerate Hessian and defines a metric on $\M$ with Lorentzian signature \cite{hitchin2000geometry,karigiannis2009hodge}. In the present work we prefer to use $\F$ as a potential, which is the convention usually adopted by physicists. In fact, both conventions agree when restricted to the moduli space $\M_1$ of torsion-free $G_2$-structures with unit volume,  but we prefer to use $\F$ since it is more convenient to work with a Riemannian metric instead of a Lorentzian one. Moreover, since $(\M,\G)$ is isometric to $\R \times (\M_1,\G_1)$ all geometric invariants of interest can be computed in $\M$, which has a natural affine structure, and directly restricted to $\M_1$, whereas it would be more difficult to do computations directly in $\M_1$ for lack of natural coordinates.

\begin{rem}     \label{rem:idgf}
    A couple of useful identities to note in local affine coordinates are
    \begin{equation*}
        x^k \G_{ak} = x^k \F_{ak} = - \F_a, ~~~ \text{and} ~~~ x^k \F_k = -7 .
    \end{equation*}
    They just follow from the fact that $x^k$ are by definition the coordinates of the cohomology class $[\varphi] \in H^3(M)$ and 
    \begin{align*}
        \G_\varphi(\varphi,\eta_a) & = \frac{1}{\Vol(\varphi)} \int \langle \varphi, \eta_a \rangle_\varphi \mu_\varphi = - \F_a, \\
        d_\varphi \F(\varphi) & = - \frac{1}{\Vol(\varphi)} \int \varphi \wedge \Theta(\varphi) = - 7 .
    \end{align*}
\end{rem}

For the purpose of computing higher derivatives of the potential, it will be convenient to adopt the following definition:

\begin{Def}
    Let $\mathscr U \subseteq \M$ be an open subset of the moduli space. A \emph{local section} of the moduli space defined on $\mathscr U$ is a smooth map $\underline \varphi : \mathscr U \times M \rightarrow \Lambda^3T^*M$, such that for any $u \in \mathscr U$ the restriction $\varphi_u = {\underline \varphi}_{\{ u \} \times M}$ is a torsion-free $G_2$-structure on $M$ and $u = \varphi_u \D$ in $\M$. A section $\underline \varphi$ is said to be \emph{adapted} at $u_0 \in \mathscr U$ if the tangent map $T_{u_0} \mathscr U \rightarrow \Omega^3(M)$ of the induced map $\mathscr U \rightarrow \Omega^3(M)$ takes values in the space $\mathscr{H}^3(M,\varphi_{u_0})$ of harmonic $3$-forms for the metric induced by $\varphi_{u_0}$. 
\end{Def}

In affine coordinates $x = (x^0, \ldots , x^n)$, where $n+1 = b^3(M)$, a local section $\underline \varphi = (\varphi_x)_x$ of the moduli space is adapted at a point $u_0$ with coordinates $x_0$ if and only if for any $0 \leq a \leq n$, the $3$-form $\left. \frac{\partial \varphi_x}{\partial x^a} \right|_{x=x_0}$ is harmonic for the metric induced by $\varphi_{x_0}$. By the proof of \cite[Th. C]{joyce1996compacti}, there exist adapted sections through any point of the moduli space. The interest of working with sections that are adapted at a point is the following lemma, which will simplify many computations:

\begin{lem}     \label{lem:diffxnot}
    Let $\mathscr U$ be an open subset of $\M$, $x = (x^0, \ldots , x^n)$ be affine coordinates on $\mathscr U$ and $x_0 \in \mathscr U$. Let $\underline \varphi = (\varphi_x)_x$ be a local section of the moduli space adapted at $x_0$, and let $f : \mathscr U \times M \rightarrow \R$ be a smooth function. Then at $x = x_0$:
    \begin{equation*}
        \left. \frac{\partial}{\partial x^a} \right|_{x=x_0} \left( \frac{1}{\Vol(\varphi_x)} \int f_x \mu_{\varphi_x} \right) = \frac{1}{\Vol(\varphi_{x_0})} \int \left. \frac{\partial f_x}{\partial x^a} \right|_{x=x_0} \mu_{\varphi_{x_0}}, ~~~ \forall a = 0, \ldots , n .
    \end{equation*}
\end{lem}

\begin{proof}
    After a linear change of coordinates, we may choose a basis $\eta_0,\ldots,\eta_n$ of $\mathscr H^3(M,\varphi_{x_0})$ such that $\eta_0 \in \mathscr H^3_1(M,\varphi_{x_0})$ and $\eta_a \in \mathscr H^3_{27}(M,\varphi_{x_0})$ for $a = 1,\ldots,n$ and assume that $(x^0,\ldots,x^n)$ are the associated local coordinates (that is, $\frac{\partial}{\partial x^a} = [\eta_a] \in H^3(M)$). In these coordinates we have:
    \begin{multline}        \label{eq:deradapted}
        \left. \frac{\partial}{\partial x^a} \right|_{x=x_0} \left( \frac{1}{\Vol(\varphi_x)} \int f_x \mu_{\varphi_x} \right) = \frac{1}{\Vol(\varphi_{x_0})} \int \left. \frac{\partial f_x}{\partial x^a} \right|_{x=x_0} \mu_{\varphi_{x_0}} \\
            + \frac{1}{\Vol(\varphi_{x_0})} \int f_{x_0} \left.\frac{\partial \mu_{\varphi_{x_0}}}{\partial x^a} \right|_{x=x_0} + \left. \frac{\partial}{\partial x^a} \left( \frac{1}{\Vol(\varphi_x)} \right) \right|_{x=x_0} \int f_{x_0} \mu_{\varphi_{x_0}} .
    \end{multline}
    Since the section is adapted at the point $x_0$, $\frac{\partial \varphi_x}{\partial x^0}$ is a harmonic section of $\Omega^3_1(M)$ and $\frac{\partial \varphi_x}{\partial x^a}$ are harmonic sections of $\Omega^3_{27}(M)$ for $a = 1, \ldots , n$ at $x=x_0$. Hence, if $a \geq 1$ then $\frac{\partial \mu_{\varphi_x}}{\partial x^a} = 0$ at $x = x_0$, which also implies $\frac{\partial \Vol(\varphi_x)}{\partial x^a} = 0$. Therefore, both terms in the second line of \eqref{eq:deradapted} vanish. For the derivative along the coordinate $x^0$, there exists $\lambda \neq 0$ such that $\frac{\partial \varphi_x}{\partial x^0} = \lambda \varphi_{x_0}$ at $x = x_0$, and by Lemma \ref{lem:firstvariationphi} this implies:
    \begin{equation*}
        \frac{\partial \mu_{\varphi_x}}{\partial x^0} = \frac{7 \lambda}{3} \mu_{\varphi_x} , ~~~ \frac{\partial}{\partial x^0} \left( \frac{1}{\Vol(\varphi_x)} \right) = - \frac{7 \lambda}{3} \frac{1}{\Vol(\varphi_x)}
    \end{equation*}
    at $x = x_0$. Therefore the lemma also holds for $a = 0$ since the two terms in the second line of \eqref{eq:deradapted} cancel each other. 
\end{proof}



   \section{Higher derivatives of the potential}       \label{section:derivatives}

In this section, we present a new derivation of the derivatives of the potential $\F$ up to order $4$, and derive a few consequences for the geometry of the moduli space. First, we study in \S\ref{subsection:harmonicdeform} the infinitesimal deformations of harmonic forms along a family of Riemannian metrics. The derivations of the third and fourth derivatives of the potential are carried out in \S\ref{subsection:34der}. In \S\ref{subsection:curvatures} we relate them to the curvatures of the moduli spaces. In \S\ref{subsection:t3k3} we push further our computations for the case of $(T^3 \times K3)/\Gamma$. Another geometric interpretation of our formulas will be given the next section.

    \subsection{Deformations of harmonic forms along a family of metrics}       \label{subsection:harmonicdeform}

In this part, we let $(M^7,g)$ be an oriented compact Riemannian $7$-manifold and $h \in \End(TM)$ be a \emph{trace-free} endomorphism, self-adjoint for the metric $g$. Moreover, let $\{g_t\}_{t \in (-\epsilon,\epsilon)}$ be a smooth family of metrics such that $g_0 = g$ and $\left. \frac{\partial g_t}{\partial t} \right|_{t = 0} = 2 g(h, \cdot)$. For all $|t| < \epsilon$, we denote by $h_t$ the unique $g_t$-self-adjoint endomorphism of $TM$ such that $\frac{\partial g_t}{\partial t} = 2 g_t(h_t \cdot, \cdot)$. In particular, $h_0 = h$, but we do not require $h_t$ to be trace-free with respect to $g_t$ for $t \neq 0$. We also denote by $*$ the Hodge operator associated with $g$, and by $d^*$ and $\Delta = (dd^* + d^*d)$ the corresponding operators; similarly for $t \in (-\epsilon,\epsilon)$ we denote by $*_t$, $d^{*_t}$ and $\Delta_t$ the operators associated with $g_t$. We want to understand the infinitesimal variations of the harmonic representative of a fixed cohomology class along the path $\{g_t\}_{t \in (-\epsilon,\epsilon)}$. We start by describing the deformations of the operator $d^{*_t}$.

\begin{lem}     \label{lem:deldstar}
    If $\eta \in \Omega^k(M)$ is a $k$-form, we have
    \begin{equation*}
        \left. \frac{\partial d^{*_t}\eta}{\partial t} \right|_{t=0} = 2 h \cdot (d^* \eta) - 2 d^* (h \cdot \eta) .
    \end{equation*}
\end{lem}

\begin{proof}
    By definition, $d^{*_t} \eta = (-1)^k *_t d *_t \eta$. Using Lemma \ref{lem:firstvariationh}, we know that
    \begin{equation*}
        \left. \frac{\partial *_t}{\partial t} \right|_{t=0} \eta = h \cdot (* \eta) - * (h \cdot \eta) = 2 h \cdot (* \eta) = - 2 * (h \cdot \eta)
    \end{equation*}
    where the last two inequalities follow from Corollary \ref{cor:hstar}, since $h$ is trace-free and self-adjoint for the metric $g$. The lemma follows.
\end{proof}

\begin{lem}     \label{lem:deltadeleta}
    Let $\{\eta_t\}_{t \in (-\epsilon,\epsilon)}$ be a smooth family of $k$-forms on $M$, such that $\eta_t$ is harmonic for the metric $g_t$ for all $|t| < \epsilon$, and let $\eta = \eta_0$. Then we have:
    \begin{equation*}
        \Delta \left. \frac{\partial \eta_t}{\partial t} \right|_{t=0} = 2 d d^* (h \cdot \eta) .
    \end{equation*}
\end{lem}

\begin{proof}
    The $k$-form $\eta_t$ is closed for all $t \in (-\epsilon, \epsilon)$, and thus if we differentiate the equality
    \begin{equation*}
        (d^{*_t} d + d d^{*_t}) \eta_t = 0
    \end{equation*}
    with respect to $t$ we obtain
    \begin{equation*}
        d \frac{\partial d^{*_t}}{\partial t} \eta_t + \Delta_t \frac{\partial \eta_t}{\partial t} = 0 .
    \end{equation*}
    At $t=0$, $h_0 = h$ is trace-free, $\eta_0 = \eta$ satisfies $d^* \eta = 0$, and thus the previous lemma yields
    \begin{equation*}
        \left. \frac{\partial d^{*_t}}{\partial t} \right|_{t=0} \eta = 2 h \cdot (d^* \eta) - 2 d^* (h \cdot \eta) = - 2 d^*(h \cdot \eta) .
    \end{equation*}
    From this it follows that
    \begin{equation*}
        - 2 dd^* (h\cdot \eta)  + \Delta \left. \frac{\partial \eta_t}{\partial t} \right|_{t=0} = 0
    \end{equation*}
    which proves our claim.
\end{proof}

In the next part we will need the following consequence of the previous lemmas:

\begin{cor}     \label{cor:decompgeta}
    Let $\eta$ be harmonic $k$-form with respect to the metric $g$. For $t \in (-\epsilon,\epsilon)$, we denote by $\eta_t$ the harmonic representative of $[\eta] \in H^k(M)$ for the metric $g_t$ and by $\nu_t$ the harmonic representative of the cohomology class $[* \eta] \in H^{7-k}(M)$. Then the decomposition of $h \cdot \eta$ into harmonic, exact and co-exact parts reads:
    \begin{equation*}
        h \cdot \eta = \mathscr H(h \cdot \eta) + \left. \frac{1}{2}\frac{\partial \eta_t}{\partial t} \right|_{t=0} - \frac{1}{2} * \left. \frac{\partial \nu_t}{\partial t} \right|_{t=0} \cdot
    \end{equation*}
\end{cor}

\begin{proof}
    By the previous lemma, $h \cdot \eta$ satisfies the equation
    \begin{equation*}
        \Delta \left. \frac{\partial \eta_t}{\partial t} \right|_{t=0} = 2 dd^* (h \cdot \eta) .
    \end{equation*}
    Moreover, as $\eta_t$ represents a fixed cohomology class, the $k$-forms $\frac{\partial \eta_t}{\partial t}$ are exact. Therefore, the exact part of $h \cdot \eta$ is $\frac{1}{2} \left. \frac{\partial \eta_t}{\partial t} \right|_{t=0}$. 

    The co-exact part of $h \cdot \eta$ can be deduced by symmetry. Indeed, as $*^2 = (-1)^{k(7-k)} = 1$ on $k$-forms, the co-exact part of $h \cdot \eta$ is the Hodge dual of the exact part of $* (h \cdot \eta)$. As $h$ is trace-free, Corollary \ref{cor:hstar} implies that $* (h \cdot \eta) = - h \cdot (* \eta)$. Using the above characterisation of the exact part, we deduce that the exact part of $h \cdot (* \eta)$ is precisely $\frac{1}{2} \left. \frac{\partial \nu_t}{\partial t} \right|_{t=0}$. Thus the co-exact part of $h \cdot \eta$ is $- \frac{1}{2} * \left. \frac{\partial \nu_t}{\partial t} \right|_{t=0}$.
\end{proof}

    \subsection{The third and fourth derivatives}       \label{subsection:34der}

In this part, $M$ is a compact oriented $7$-manifold with $b^1(M) = 0$ admitting torsion-free $G_2$-structures, and we aim to compute the third and fourth derivative of the potential $\F$. Using a basis $u_0,\ldots,u_n$ of $H^3(M)$, $n=b^3_{27}(M)=b^3(M)-1$, we define affine coordinates $x = (x^0,\ldots,x^n)$ on $\M$. If $\varphi$ is a torsion-free $G_2$-structures on $M$, we denote by $\eta_a \in \Omega^3(M)$ the unique $g_\varphi$-harmonic representative of the cohomology class $u_a \in H^3(M)$, and by $h_a \in C^\infty(\End(TM))$ the unique endomorphism orthogonal to $\Omega^2_{14}(M)$ such that $h_a \cdot \varphi = \eta_a$. Since $b^1(M) = 0$, the $3$-form $\eta_a$ has no $\Omega^3_7$-component, and thus $h_a$ is self-adjoint with respect to the metric $g_\varphi$. Similarly, if $\{\varphi_x\}$ is a local section of the moduli space, we denote by $\eta_{a,x} \in \Omega^3(M)$ and by $h_{a,x} \in C^\infty(\End(TM))$ the tensors associated with $\varphi_x$.

Various formulas for the third derivative of the potential already appear is the literature \cite{grigorian2009local,grigorian2010moduli,karigiannis2009hodge,lee2009geometric}. Here we give an independent derivation:

\begin{prop}        \label{prop:fabc}
    Let $\varphi$ be a torsion-free $G_2$-structure on $M$. Then the third derivative of the potential satisfies:
    \begin{equation*}
        \F_{abc}(\varphi \D) = - \frac{2}{\Vol(\varphi)} \int \langle h_c \cdot \eta_a, \eta_b \rangle_\varphi \mu_\varphi \cdot
    \end{equation*}
\end{prop}

\begin{proof}
    Let $x = (x^0,\ldots,x^n)$ be local affine coordinates on $\M$, let $x_0$ be the coordinates of $\varphi$, and let $\{\varphi_x\}$ be a local adapted section of the moduli space through $\varphi$. Differentiating the identity (which comes from \eqref{eq:fabetaab})
    \begin{equation*}
        \F_{ab}(\varphi_x \D) = \frac{1}{\Vol(\varphi_x)}\int \langle \eta_{a,x}, \eta_{b,x} \rangle_{\varphi_x} \mu_{\varphi_x}
    \end{equation*}
    and using Lemma \ref{lem:diffxnot}, we obtain at $x=x_0$:
    \begin{multline*}
        \F_{abc}(\varphi \D) = \frac{1}{\Vol(\varphi)} \int \left. \frac{\partial g_{\varphi_x}}{\partial x^c} \right|_{x=x_0} (\eta_{a}, \eta_{b}) \mu_\varphi  \\ 
            + \frac{1}{\Vol(\varphi)} \int \langle \left. \frac{\partial \eta_{a,x}}{\partial x^c} \right|_{x=x_0} , \eta_b \rangle_\varphi \mu_\varphi + \frac{1}{\Vol(\varphi)} \int \langle \eta_a, \left. \frac{\partial \eta_{b,x}}{\partial x^c} \right|_{x=x_0} \rangle_\varphi \mu_\varphi .
    \end{multline*}
    The $3$-forms $\frac{\partial \eta_{a,x}}{\partial x^c}$ and $\frac{\partial \eta_{b,x}}{\partial x^c}$ are exact since $\eta_{a,x}$ and $\eta_{b,x}$ represent constant cohomology classes, and therefore the second and third terms above vanish. On the other hand, as the section $\{\varphi_x\}$ is adapted at $x=x_0$, we have $\left. \frac{\partial \varphi_x}{\partial x^c} \right|_{x=x_0} = \eta_c = h_c \cdot \varphi$. Thus we can compute the first term using Lemma \ref{lem:firstvariationphi} and the fact that $h_c$ is self-adjoint with respect to $g_\varphi$:
    \begin{equation*}
        \F_{abc}(\varphi \D) = - \frac{1}{\Vol(\varphi)} \int ( \langle h_c \cdot \eta_a, \eta_b \rangle_\varphi + \langle \eta_a, h_c \cdot \eta_b \rangle_\varphi ) \mu_\varphi = - \frac{2}{\Vol(\varphi)} \int \langle h_c \cdot \eta_a, \eta_b \rangle_\varphi \mu_\varphi
    \end{equation*}
    at $x=x_0$.
\end{proof}

We now proceed with the derivation of the fourth derivative. As a first step, we prove a formula which depends on a particular choice of local section of the moduli space:

\begin{prop}        \label{prop:fabcdi}
    Let $\varphi$ be a torsion-free $G_2$-structure, let $\{\varphi_x\}$ be a local adapted section of the moduli space through $\varphi$ and denote by $x=x_0$ the coordinates of $\varphi \D$. Then the fourth derivative of the potential satisfies:
    \begin{align*}
        \F_{abcd}(\varphi \D) = ~~& \frac{2}{\Vol(\varphi)} \int \langle h_d \cdot \eta_a - \left. \frac{\partial \eta_{a,x}}{\partial x^d} \right|_{x=x_0}, h_c \cdot \eta_b \rangle_\varphi \mu_\varphi \\
            + & \frac{2}{\Vol(\varphi)} \int \langle h_d \cdot \eta_b - \left. \frac{\partial \eta_{b,x}}{\partial x^d} \right|_{x=x_0}, h_c \cdot \eta_a \rangle_\varphi \mu_\varphi \\
            + & \frac{2}{\Vol(\varphi)} \int \langle h_d \cdot \eta_c - \left. \frac{\partial \eta_{c,x}}{\partial x^d} \right|_{x=x_0}, h_a \cdot \eta_b \rangle_\varphi \mu_\varphi .
    \end{align*}
\end{prop}

\begin{proof}
    To lighten notations, we will keep the $x$-dependence implicit and write $\eta_a$ and $h_a$ instead of $\eta_{a,x}$ and $h_{a,x}$ when this does not create any confusion. Also, unless otherwise noted we differentiate at $x=x_0$. By the previous proposition, the third derivative of the potential can be written:
    \begin{equation*}
        \F_{abc}(\varphi_x \D) = - \frac{1}{\Vol(\varphi_x)} \int \langle h_c \cdot \eta_a, \eta_b \rangle_{\varphi_x} \mu_{\varphi_x} - \frac{1}{\Vol(\varphi_x)} \int \langle \eta_a, h_c \cdot \eta_b \rangle_{\varphi_x} \mu_{\varphi_x} .
    \end{equation*}
    Differentiating with respect to $x^d$ at $x=x_0$ and using Lemma \ref{lem:diffxnot} we obtain:
    \begin{align}       \label{eq:fourthder}
        \begin{split}
            \F_{abcd}(\varphi \D) = & - \frac{1}{\Vol(\varphi)} \int \frac{\partial g_{\varphi_x}}{\partial x^d}(h_c \cdot \eta_a, \eta_b ) \mu_\varphi - \frac{1}{\Vol(\varphi)} \int \frac{\partial g_{\varphi_x}}{\partial x^d}( \eta_a, h_c \cdot \eta_b ) \mu_\varphi \\
            & - \frac{1}{\Vol(\varphi)} \int \langle h_c \cdot \eta_a, \frac{\partial \eta_{b,x}}{\partial x^d} \rangle_\varphi \mu_\varphi - \frac{1}{\Vol(\varphi)} \int \langle \frac{\partial \eta_{a,x}}{\partial x^d} , h_c \cdot  \eta_b \rangle_\varphi \mu_\varphi \\
            & - \frac{1}{\Vol(\varphi)} \int \langle h_c \cdot \frac{\partial \eta_{a,x}}{\partial x^d} , \eta_b \rangle_\varphi \mu_\varphi - \frac{1}{\Vol(\varphi)} \int \langle \eta_a, h_c \cdot \frac{\partial \eta_{b,x}}{\partial x^d} \rangle_\varphi \mu_\varphi \\
            & - \frac{1}{\Vol(\varphi)} \int \langle \frac{\partial h_{c,x}}{\partial x^d} \cdot \eta_a, \eta_b \rangle_\varphi \mu_\varphi - \frac{1}{\Vol(\varphi)} \int \langle \eta_a , \frac{\partial h_{c,x}}{\partial x^d} \cdot \eta_b \rangle_\varphi \mu_\varphi \cdot
        \end{split}
    \end{align}
    Since the section $\{\varphi_x\}$ is adapted, at $x=x_0$ we have $\frac{\partial \varphi_x}{\partial x^d} = \eta_d = h_d \cdot \varphi$ and by Lemma \ref{lem:firstvariationphi} we have the identities:
    \begin{equation*}
        \frac{\partial g_{\varphi_x}}{\partial x^d}(h_c \cdot \eta_a, \eta_b ) = - 2 \langle h_c \cdot \eta_a, h_d \cdot \eta_b \rangle_\varphi, ~~~  \frac{\partial g_{\varphi_x}}{\partial x^d}(\eta_a, h_c \cdot \eta_b ) = - 2 \langle h_d \cdot \eta_a, h_c \cdot \eta_b \rangle_\varphi .
    \end{equation*}
    Moreover, since the section $h_c$ of $\End(TM)$ is self-adjoint for the metric induced by $\varphi$, the second and third lines in \eqref{eq:fourthder} are equal. These observations yield:
    \begin{align}       \label{eq:fourthderb}
        \begin{split}
            \F_{abcd}(\varphi \D) = ~~ & \frac{2}{\Vol(\varphi)} \int \langle h_d \cdot \eta_a - \frac{\partial \eta_{a,x}}{\partial x^d}, h_c \cdot \eta_b \rangle_\varphi \mu_\varphi \\
            + & \frac{2}{\Vol(\varphi)} \int \langle h_d \cdot \eta_b - \frac{\partial \eta_{b,x}}{\partial x^d}, h_c \cdot \eta_a \rangle_\varphi \mu_\varphi \\
            - & \frac{1}{\Vol(\varphi)} \int \langle \frac{\partial h_{c,x}}{\partial x^d} \cdot \eta_a, \eta_b \rangle_\varphi + \langle \eta_a , \frac{\partial h_{c,x}}{\partial x^d} \cdot \eta_b \rangle_\varphi \mu_\varphi .
        \end{split}  
    \end{align}
    It remains to show that the last line in \eqref{eq:fourthderb} can be put in a form similar to the first two lines. Decomposing $\frac{\partial h_{c,x}}{\partial x^d}$ into $g_\varphi$-self-adjoint and $g_\varphi$-anti-self-adjoint parts, we can further write:
    \begin{align*}
        \langle \frac{\partial h_{c,x}}{\partial x^d} \cdot \eta_a, \eta_b \rangle_\varphi + \langle \eta_a , \frac{\partial h_{c,x}}{\partial x^d} \cdot \eta_b \rangle_\varphi & = \langle \left( \frac{\partial h_{c,x}}{\partial x^d} + \left( \frac{\partial h_{c,x}}{\partial x^d} \right)^{\dag_\varphi} \right) \cdot \eta_a , \eta_b \rangle_\varphi \\
        & = \langle \left( \frac{\partial h_{c,x}}{\partial x^d} + \left( \frac{\partial h_{c,x}}{\partial x^d} \right)^{\dag_\varphi} \right) \cdot \varphi , h_a \cdot \eta_b \rangle_\varphi 
    \end{align*}
    where the second equality follows from Corollary \ref{cor:symyukawa} and $\left( \frac{\partial h_{c,x}}{\partial x^d} \right)^{\dag_\varphi}$ denotes the adjoint of $\frac{\partial h_{c,x}}{\partial x^d}$ with respect to the metric $g_\varphi$. Taking the self-adjoint part of a section $h$ of $\End(TM)$ corresponds to projecting $h \cdot \varphi$ onto the $\Omega^3_1 \oplus \Omega^3_{27}$-components, and hence we obtain:
    \begin{equation*}
        \langle \frac{\partial h_{c,x}}{\partial x^d} \cdot \eta_a, \eta_b \rangle_\varphi + \langle \eta_a , \frac{\partial h_{c,x}}{\partial x^d} \cdot \eta_b \rangle_\varphi = 2 \langle \frac{\partial h_{c,x}}{\partial x^d} \cdot \varphi , \pi_{1 \oplus 27}(h_a \cdot \eta_b) \rangle_\varphi .
    \end{equation*}
    Differentiating the relation $h_{c,x} \cdot \varphi_x = \eta_{c,x}$ at $x=x_0$ gives $\frac{\partial h_{c,x}}{\partial x^d} \cdot \varphi_x = \frac{\partial \eta_{c,x}}{\partial x^d} - h_c \cdot \eta_d$ and thus:
    \begin{align*}
        2 \langle \frac{\partial h_{c,x}}{\partial x^d} \cdot \varphi , \pi_{1 \oplus 27}(h_a \cdot \eta_b) \rangle_\varphi & = - 2 \langle h_c \cdot \eta_d - \frac{\partial \eta_{c,x}}{\partial x^d}, \pi_{1 \oplus 27}(h_a \cdot \eta_b) \rangle_\varphi \\
        & = - 2 \langle h_d \cdot \eta_c - \frac{\partial \eta_{c,x}}{\partial x^d}, \pi_{1 \oplus 27}(h_a \cdot \eta_b) \rangle_\varphi
    \end{align*}
    where the second equality also holds because this expression is invariant under permutation of $h_c$ and $h_d$. It remains to prove that the component $\pi_7(h_d \cdot \eta_c - \frac{\partial \eta_{c,x}}{\partial x^d})$ vanishes. This component can be singled out by wedging with $\varphi$. On the one hand, we have:
    \begin{equation*}
        (h_d \cdot \eta_c ) \wedge \varphi = h_d \cdot (\eta_c \wedge \varphi ) - \eta_c \wedge (h_d \cdot \varphi) = - \eta_c \wedge \eta_d 
    \end{equation*}
    as $\eta_c \wedge \varphi = 0$ since $\pi_7(\eta_c) = 0$. On the other hand, at $x=x_0$ we can write
    \begin{equation}        \label{eq:delhetaphi}
        \frac{\partial \eta_{c,x}}{\partial x^d} \wedge \varphi_x = \frac{\partial}{\partial x^d}( \eta_{c,x} \wedge \varphi_x ) - \eta_c \wedge \frac{\partial \varphi_x}{\partial x^d} = - \eta_c \wedge \eta_d 
    \end{equation}
    since $\frac{\partial \varphi_x}{\partial x^d} = \eta_d$ at $x=x_0$. Therefore $\pi_7(h_d \cdot \eta_c - \frac{\partial \eta_c}{\partial x^d}) = 0$. Putting everything together this implies that, at $x=x_0$:
    \begin{equation*}
        \langle \frac{\partial h_{c,x}}{\partial x^d} \cdot \eta_a, \eta_b \rangle_\varphi + \langle \eta_a , \frac{\partial h_{c,x}}{\partial x^d} \cdot \eta_b \rangle_\varphi = - 2 \langle h_d \cdot \eta_c - \frac{\partial \eta_{c,x}}{\partial x^d}, h_a \cdot \eta_b \rangle_\varphi
    \end{equation*}
    which yields the claimed expression for $\F_{abcd}(\varphi \D)$.
\end{proof}

The above expression for $\F_{abcd}$ is unsatisfactory, as it involves choosing an adapted section at a point of $\M$. In order to rewrite it in a more intrinsic way, we need to decompose the $3$-forms $h_d \cdot \eta_a$, $h_d \cdot \eta_b$ and $h_d \cdot \eta_c$ using the results of the previous section:

\begin{lem}     \label{lem:hetadecomp}
    With the notations of the previous proposition, the decomposition of $h_d \cdot \eta_c$ into harmonic, exact and co-exact parts reads:
    \begin{equation*}
        h_d \cdot \eta_c = \mathscr H(h_d \cdot \eta_c) + \frac{1}{2} \left. \frac{\partial \eta_c}{\partial x^d} \right|_{x=x_0} - \frac{1}{2}*_\varphi \left. \frac{\partial \nu_c}{\partial x^d} \right|_{x=x_0}
    \end{equation*}
    where $\nu_{c,x}$ is the harmonic representative of the cohomology class $[*_\varphi \eta_c] \in H^4(M)$ for the metric induced by $\varphi_x$.
\end{lem}

\begin{proof}
    After applying a linear change of coordinates if necessary, we may assume that at $x = x_0$ the harmonic form $\eta_0$ is proportional to $\varphi$ and $\eta_1, \ldots, \eta_n$ are in $\mathscr H^3_{27}(M,\varphi)$. Thus if $d = 0$, $h_d \in C^{\infty}(\End(TM))$ is a constant multiple of the identity, and therefore $h_d \cdot \eta_c$ is harmonic. Moreover, variations of $\varphi$ in the direction $\eta_0$ correspond to scaling the $G_2$-structure, and the harmonic representatives of a fixed cohomology class are constant under scaling of the metric. Therefore the proposition holds if $d = 0$. On the other hand, if $d = 1,\ldots,n$ then the result follows from Corollary \ref{cor:decompgeta}.
\end{proof}

As a consequence of this lemma, we can write with the notations of Proposition \ref{prop:fabcdi}
\begin{equation*}
    h_d \cdot \eta_c - \left. \frac{\partial \eta_{c,x}}{\partial x^d} \right|_{x=x_0} = \mathscr{H}(h_d \cdot \eta_c) + G_\Delta((d^*d-dd^*)(h_d \cdot \eta_c))
\end{equation*}
where $G_\Delta$ denotes the Green's function of the Laplacian (acting on the orthogonal component of the space of harmonic forms) associated with $g_\varphi$. Moreover, we can use Proposition \ref{prop:fabc} to decompose the harmonic $3$-form $\mathscr H(h_d \cdot \eta_c)$ in the basis $\eta_0, \cdots,\eta_n$ as:
\begin{equation*}
    \mathscr H(h_d \cdot \eta_c) = \frac{\G^{kl}}{\Vol(\varphi)} \int \langle h_d \cdot \eta_c, \eta_k \rangle_\varphi \mu_\varphi \cdot \eta_l = - \frac{1}{2} \G^{kl} \F_{cdk} \eta_l 
\end{equation*}
and thus
\begin{equation*}
    \frac{2}{\Vol(\varphi)} \int \langle \mathscr H (h_d \cdot \eta_c) , \mathscr H (h_a \cdot \eta_b) \rangle_\varphi \mu_\varphi = \frac{1}{2} \G^{kl} \F_{abk}\F_{cdl} .
\end{equation*}
Therefore, we obtain a formula which does not depend on any choice of local section: 

\begin{thm}         \label{thm:fabcd}
    The fourth derivative of the potential is given by
    \begin{equation*}
    \mathscr F_{abcd} = \frac 1 2 \G^{kl} \left( \F_{abk} \F_{cdl} + \F_{ack}\F_{bdl} + \F_{adk} \F_{bcl} \right) + \mathscr{E}_{abcd} + \mathscr{E}_{cabd} + \mathscr{E}_{cbad}
    \end{equation*}
    where for any torsion-free $G_2$-structure $\varphi$ on $M$ we have
    \begin{equation*}
        \mathscr{E}_{abcd}(\varphi \D) = \frac{2}{\Vol(\varphi)} \int \langle G_\Delta((d^*d-dd^*) h_d \cdot \eta_c), h_a \cdot \eta_b \rangle_\varphi \mu_\varphi .
    \end{equation*}
\end{thm}

\begin{rem}     \label{rem:exactparts}
    Said in words, the integral $\int \langle G_\Delta((d^*d-dd^*) h_d \cdot \eta_c), h_a \cdot \eta_b \rangle_\varphi \mu_\varphi$ is the $L^2$-inner product of the co-exact parts of $h_a \cdot \eta_b$ and $h_d \cdot \eta_c$ minus the $L^2$-inner product of the exact parts of $h_a \cdot \eta_b$ and $h_d \cdot \eta_c$. Hence the term $\mathscr{E}_{abcd}$ vanishes exactly when these inner products are equal.
\end{rem}

\begin{rem}
    Since the operator $G_\Delta(d^*d-dd^*)$ is self-adjoint we have $\mathscr{E}_{abcd} = \mathscr{E}_{dcba}$. A slightly less obvious symmetry is the fact that $\mathscr{E}_{abcd} = \mathscr{E}_{bacd}$, which we can prove in two ways. The first comes from the symmetry of the partial derivatives of $\F$, which implies that $\mathscr{E}_{abcd} + \mathscr{E}_{cabd} + \mathscr{E}_{cbad}$ is fully symmetric in its indices. The sum of the last two terms is symmetric under permutations of $a$ and $b$, and hence the first term $\mathscr{E}_{abcd}$ must be symmetric in the indices $a$ and $b$. As a sanity check, we can also recover this symmetry property from the expression given in Theorem \ref{thm:fabcd}. Indeed we can deduce from Lemma \ref{lem:derivation} the expression
    \begin{equation*}
        \mathscr{E}_{abcd} - \mathscr{E}_{bacd} = \frac{2}{\Vol(\varphi)} \int \langle G_\Delta((d^*d-dd^*) h_d \cdot \eta_c), [h_b,h_a] \cdot \varphi \rangle_\varphi \mu_\varphi .
    \end{equation*}
    Now $[h_b,h_a]$ is anti-self-adjoint for the metric $g_\varphi$ and hence the $3$-form $[h_a,h_b]$ is of type $\Omega^3_7$. In particular, it is orthogonal to the space of harmonic $3$-forms, and hence Lemma \ref{lem:hetadecomp} implies that if we choose a section $\varphi_x$ of the moduli space adapted at $x = x_0$ we have
    \begin{equation*}
        \mathscr{E}_{abcd}(x_0) - \mathscr{E}_{bacd}(x_0) = \frac{2}{\Vol(\varphi)} \int \langle h_d \cdot \eta_c - \left. \frac{\partial \eta_{c,x}}{\partial x^d} \right|_{x=x_0}, [h_b,h_a] \cdot \varphi \rangle_\varphi \mu_\varphi .
    \end{equation*}
    In the proof of Proposition \ref{prop:fabcdi}, we showed that $\pi_7(h_d \cdot \eta_c -  \left. \frac{\partial \eta_{c,x}}{\partial x^d} \right|_{x=x_0}) = 0$ which means that the expression under the integral vanishes identically for type reasons. Hence we recover the fact that $\mathscr{E}_{abcd} = \mathscr{E}_{bacd}$.

    Besides the above symmetries (and the ones we can deduce from them), there is no reason to think that $\mathscr{E}_{abcd}$ is fully symmetric in its indices; only the combination of the terms $\mathscr{E}_{abcd} + \mathscr{E}_{cabd} + \mathscr{E}_{cbad}$ is.
\end{rem}

    \subsection{Yukawa coupling and curvatures}     \label{subsection:curvatures}

In this part, we want to interpret the expressions of the third and fourth derivatives of the potential in geometric terms and relate them to the curvatures of the moduli spaces. Let us denote by $\nabla$ the flat connection coming from the local diffeomorphism $\pi : \M \rightarrow H^3(M)$ and $\nabla^\G$ the Levi-Civita of the metric $\G$. Then there is a unique matrix-valued $1$-form $\gamma$ on $\M$, called the \emph{difference tensor} of the Hessian structure $(\nabla,\G)$, such that $\nabla^\G = \nabla + \gamma$. In local affine coordinates $x=(x^0, \ldots , x^n)$, the difference tensor can be written as
\begin{equation*}
    \gamma = \Gamma^k_{ab} dx^a dx^b \otimes \frac{\partial}{\partial x^k}
\end{equation*}
where $\Gamma^k_{ab}$ are the Christoffel symbols of the metric $\G$. As the metric is the Hessian of $\F$ in affine coordinates, the Christoffel symbols read:
\begin{equation}        \label{eq:gchristsymb}
    \Gamma^k_{ab} = \frac{1}{2} \G^{kl} \F_{abl} .
\end{equation}
In particular, the difference tensor $\gamma$ is dual to the symmetric cubic form
\begin{equation*}
    \Xi = \frac{1}{2} \mathscr{F}_{abc} dx^a dx^b dx^c .
\end{equation*}
The cubic form $\Xi$ is often called the \emph{Yukawa coupling} of $\M$ \cite{grigorian2009local,karigiannis2009hodge,lee2009geometric}. The covariant derivative of the Yukawa coupling is given by:
\begin{align}   \label{eq:covderyuk}
    \begin{split}
        \nabla^\G_d \Xi_{abc} & = \partial_d \Xi_{abc} - \Gamma^{k}_{da} \Xi_{kbc} - \Gamma^{k}_{db} \Xi_{akc} - \Gamma^{k}_{dc} \Xi_{abk} \\
        & = \frac{1}{2} \F_{abcd} - \frac 1 4 \G^{kl} \left( \F_{abk} \F_{cdl} + \F_{ack}\F_{bdl} + \F_{adk} \F_{bcl} \right) .
    \end{split}
\end{align}
Hence, Theorem \ref{thm:fabcd} implies that:
\begin{equation}        \label{eq:nablaxiabcd}
    \nabla^\G_d \Xi_{abc} = \frac{1}{2} ( \mathscr{E}_{abcd} + \mathscr{E}_{cabd} + \mathscr{E}_{cbad} ) .
\end{equation}
Therefore, $\mathscr{E}_{abcd} + \mathscr{E}_{cabd} + \mathscr{E}_{cbad} = 0$ at a point for any $a,b,c,d$ if and only if the covariant derivative (with respect to the Levi-Civita connection of $\G$) of the Yukawa coupling $\Xi$, or equivalently of the difference tensor $\gamma$, vanishes at this point. For later use, we gather a few properties of the Yukawa coupling and its covariant derivative:

\begin{lem}     \label{lem:propyukawa}
    The Yukawa coupling satisfies the following properties:
    \begin{enumerate}[(i)]
        \item Under the identification $\M \simeq \R \times \M_1$, $\Xi = - dt \otimes \G + \Xi_1$ where $\Xi_1$ is the restriction of $\Xi$ to $\M_1$.
        \item In local affine coordinates, $x^k \mathscr{F}_{abk} = - 2 \G_{ab}$.
        \item $\nabla^\G\Xi$ is a fully symmetric quartic form on $T\M$.
        \item In local affine coordinates, $x^k \nabla^\G_a \Xi_{bck} =  0$. 
    \end{enumerate}
\end{lem}

\begin{proof}
    Properties (i) and (ii) are essentially equivalent since $\F_{abc} = 2 \Xi_{abc}$. Moreover (ii) can be seen from the observation that $x^k$ are the coordinates of the cohomology class $[\varphi] \in H^3(M)$, and thus
    \begin{equation*}
        x^k \F_{abk} = - \frac{2}{\Vol(\varphi)} \int \langle h_a \cdot \eta_b , \varphi \rangle_\varphi \mu_\varphi = - \frac{2}{\Vol(\varphi)} \int \langle \eta_a, \eta_b \rangle_\varphi \mu_\varphi = -2 \G_{ab}
    \end{equation*}
    using the symmetry of $h_a$ and the fact that $h_a \cdot \varphi = \eta_a$ by definition.

    For point (iii), the symmetry of $\nabla^\G \Xi$ follows from the symmetry of the partial derivatives of $\F$ and \eqref{eq:covderyuk}. Finally, because point (iv) will be a key argument in the proof of Theorem \ref{thm:totgeod} below, we shall give it two proofs.
    
    The first proof follows from the observation that
    \begin{equation}    \label{eq:xkeabcd}
        x^k \mathscr{E}_{kbcd} = x^k \mathscr{E}_{akcd} = x^k \mathscr{E}_{abkd} = x^k \mathscr{E}_{abck} = 0
    \end{equation}
    for any $a,b,c,d$. Indeed, from the expression given in Theorem \ref{thm:fabcd}, we have
    \begin{equation*}
        x^k \mathscr{E}_{kbcd} = \int \langle G_\Delta((d^*d-dd^*) h_d \cdot \eta_c), (x^k h_k \cdot \eta_b) \rangle_\varphi \mu_\varphi .
    \end{equation*}
    Notice that $x^kh_k$ is a self-adjoint endomorphism of $TM$ for the metric $g_\varphi$, and moreover $x^k h_k \cdot \varphi = x^k \eta_k = \varphi$. It follows that $x^k h_k = \frac{1}{3} \Id$. Hence $x^k h_k \cdot \eta_b = \frac{1}{3} \Id \cdot \eta_b = \eta_b$ is harmonic, whence it is $L^2$-orthogonal to $G_\Delta((d^*d-dd^*) h_d \cdot \eta_c)$ and therefore $x^k \mathscr{E}_{kbcd} = 0$. The other identities of \eqref{eq:xkeabcd} are proved in the same way, since $x^k h_a \cdot \eta_k = h_a \cdot \varphi = \eta_a$, $x^k h_d \cdot \eta_k = \eta_d$ and $x^k h_k \cdot \eta_c = \eta_c$ are all harmonic forms. Since $a,b,c,d$ are arbitrary, point (iv) now follows from \eqref{eq:nablaxiabcd}.

    We can also give a second proof which does not rely on the particular expression of $\mathscr{E}_{abcd}$ given in Theorem \ref{thm:fabcd} but only on the properties of the potential $\F$. The idea is to differentiate the expression of point (ii) with respect to the variable $x^c$, which yields the identity $x^k \F_{abck} + \F_{abc} = - 2 \F_{abc}$, that is:
    \begin{equation}        \label{eq:xkff}
        x^k \F_{abck} = - 3 \F_{abc} .
    \end{equation}
    On the other hand, using \eqref{eq:covderyuk} we have
    \begin{equation*}
        x^k \nabla^\G_a \Xi_{bck} = \frac{1}{2} x^k \F_{abck} - \frac{1}{4} \G^{rs}\left( \F_{abr} \cdot  x^k\F_{cks} + \F_{acr} \cdot x^k\F_{bks} + x^k\F_{akr}\cdot \F_{bcs} \right)
    \end{equation*}
    and using point (ii) again we see that
    \begin{equation*}
        \G^{rs} x^k\F_{cks} = - 2 \G^{rs} \G_{cs} = - 2 \delta^r_c, ~~ \G^{rs} x^k\F_{bks} = - 2 \delta^r_b, ~~ \G^{rs}x^k\F_{akr} = -2 \delta^r_a .
    \end{equation*}
    Substituting this into the previous expression, we obtain
    \begin{equation*}
        x^k \nabla^\G_a \Xi_{bck} = \frac{1}{2} (x^k \F_{abck} + 3 \F_{abc}) = 0
    \end{equation*}
    because of \eqref{eq:xkff}.
\end{proof}

It is interesting to relate the previous observations to the curvature of $\G$. By convention, we define the Riemann curvature tensor of $\mathscr \G$ as
\begin{equation*}
    \mathscr R \left( \frac{\partial}{\partial x^c}, \frac{\partial}{\partial x^d} \right) \frac{\partial}{\partial x^b} = {\mathscr R^a}_{bcd} \frac{\partial}{\partial x^a} = \nabla^\G_{\partial_c} \nabla^\G_{\partial_d} \frac{\partial}{\partial x^b} - \nabla^\G_{\partial_d} \nabla^\G_{\partial_c} \frac{\partial}{\partial x^b} \cdot
\end{equation*}
Lowering the first index, we also write
\begin{equation*}
    \mathscr R_{abcd} = \G_{ak} {\mathscr R^k}_{bcd} .
\end{equation*}
For Hessian metrics, the Riemann curvature tensor has a particularly simple expression \cite[Prop. 2.3]{shima2007geometry}:
\begin{equation}
    \mathscr R_{abcd} = \frac{1}{4} \G^{kl} (\F_{adk} \F_{bcl} - \F_{ack}\F_{bdl}) = \G^{kl} \Xi_{adk} \Xi_{bcl} - \G^{kl}\Xi_{ack}\Xi_{bdl} .
\end{equation}
Since the Yukawa coupling determines the curvature, we deduce the following:

\begin{prop}        \label{prop:eabcdvanish}
    If $\mathscr{E}_{abcd} + \mathscr{E}_{cabd} + \mathscr{E}_{cbad} = 0$ for any $0 \leq a,b,c,d \leq n$ at a point of the moduli space, then the covariant derivative of $\mathscr R$ vanishes at this point. In particular, if $\mathscr{E}_{abcd} + \mathscr{E}_{cabd} + \mathscr{E}_{cbad}$ vanishes identically on the moduli space, then $(\M,\G)$ is locally symmetric.
\end{prop}

A simple case where this condition is satisfied, and a good sanity check for the formula of Theorem \ref{thm:fabcd}, is when $M = T^7 / \Gamma$ is the quotient of a flat torus by a group of affine isometries, which we assume to be chosen such that $b^1(M) = 0$. In that case the moduli space of torsion-free $G_2$-structures on $M$ can be identified with the space $(\Lambda^3_+ \R^*_7)^\Gamma$ of $\Gamma$-invariant positive forms on $\R^7$, and the associated metrics are induced by $\Gamma$-invariant inner products. In particular, the space of harmonic forms does not change along deformations of the flat metric and therefore Lemma \ref{lem:hetadecomp} shows that $h_a \cdot \eta_b$ is harmonic which implies that all the terms $\mathscr{E}_{abcd}$ vanish in Theorem \ref{thm:fabcd}. Hence Proposition \ref{prop:eabcdvanish} implies that the moduli space is locally symmetric. This fact can be independently proved by examining the geometry of the space of positive forms $\Lambda^3_+\R^*_7$. This is a homogeneous space diffeomorphic to $GL_+(7) / G_2$, and the map $g : \Lambda^3_+\R^*_7 \rightarrow S^2_+\R^*_7$ is a homogeneous fibration (where $S^2_+\R^*_7$ is the space of inner products on $\R^7$). It is easy to see that we can endow $\Lambda^3_+\R^*_7$ and $S^2_+\R^*_7$ with homogeneous metrics in such a way that $g$ is a Riemannian fibration, and if we identify $\Lambda^3_+\R^*_7$ with the moduli space of torsion-free $G_2$-structures on $T^7$ then the metric $\G$ coincides with the homogeneous metric. Moreover $S^2_+\R^*_7$ is a symmetric space, and the vertical space of the fibration at $\varphi \in \Lambda^3_+\R^*_7$ is $\Lambda^3_{7,\varphi}$. Now if $\Gamma$ is a finite subgroup of $G_2$ which does not fix any nonzero vector in $\R^7$, then the induced map $g : (\Lambda^3_+\R^*_7)^\Gamma \rightarrow (S^2_+\R^*_7)^\Gamma$ is a covering map since $(\Lambda^3_{7,\varphi})^\Gamma \simeq (\R^7)^\Gamma = 0$ for any $\varphi \in (\Lambda^3_+\R^*_7)^\Gamma$. Hence each connected component of $(\Lambda^3_+\R^*_7)^\Gamma$ must be isometric to $(S^2_+\R^*_7)^\Gamma$, which is a totally geodesic subspace of $S^2_+\R^*_7$ and hence also itself a symmetric space.

In the next part, we will see that the Yukawa coupling is also parallel when $M = (T^3 \times K3)/ \Gamma$. In that case, the situation is more complicated since the space of harmonic forms does vary along deformations of the torsion-free $G_2$-structure, and we have to prove that the contribution of the co-exact part of the $3$-forms $h_a \cdot \eta_b$ to $\mathscr{E}_{abcd}$ compensates the contribution of their exact part. As in the case of $T^7/\Gamma$, the fact that the $G_2$-moduli space is locally symmetric when $M = (T^3 \times K3) / \Gamma$ (with our usual assumption that $b^1(M) = 0$) can be proved in a more direct way and is related to the fact that the moduli space of hyperk\"ahler metrics on the K3 surface is locally symmetric (see the author's PhD thesis for more details \cite[Ch. 6]{langlais2025phd}).

Beyond these cases, there is no reason to think that the Yukawa coupling will be a parallel tensor, because the constraints it imposes on $\G$ are too strong. Therefore, much of the difficulty in further analysing the geometric properties of the moduli spaces lies in the fact that the terms $\mathscr{E}_{abcd}$ cannot be computed more explicitly in local coordinates. In Section \ref{section:periods}, we will propose a more geometric interpretation for the presence of these terms, and prove a stronger version of Proposition \ref{prop:eabcdvanish} which shows that if they vanish then the sectional curvature of $\mathscr{G}$ is nonpositive. An interesting question to ask is whether there is always an upper bound on the curvatures of the moduli space, in relation with similar conjectures about the geometry of K\"ahler cones \cite{wilson2004sectional}. This is currently being investigated by Karigiannis and Loftin \cite{karigiannis2025octonionic}.

    \subsection{Further computations for $(T^3 \times K3)/\Gamma$}      \label{subsection:t3k3}

Let $T^3 = \R^3 / \Lambda$ for some lattice $\Lambda \subset \R^3$, and $X$ be the smooth $4$-manifold underlying K3 surfaces. If $\underline \omega = (\omega_1,\omega_2,\omega_3)$ is a hyperk\"ahler triple on $X$ with associated hyperk\"ahler metric $g_{\underline \omega}$ and $(\theta_1,\theta_2,\theta_3)$ are linear coordinates on $\R^3$, then
\begin{equation}
    \varphi_{\underline \omega} = d\theta_1 \wedge d\theta_2 \wedge d\theta_3 - d\theta_1 \wedge \omega_1 - d\theta_2 \wedge \omega_2 - d\theta_3 \wedge \omega_3 
\end{equation}
is a torsion-free $G_2$-structure on $T^3 \times X$, with associated metric
\begin{equation*}
    g_{\varphi_{\underline \omega}} = d\theta_1^2 + d\theta_2^2 + d\theta_3^2 + g_{\underline \omega} .
\end{equation*}
The space of harmonic $3$-forms on $T^3 \times X$ decomposes as:
\begin{equation*}
    \mathscr H^3(T^3 \times X, \varphi_{\underline \omega}) = \Lambda^3 \R^*_3 \oplus (\R^*_3 \otimes \mathscr H^+_{\underline \omega}(X)) \oplus (\R^*_3 \otimes \mathscr H^-_{\underline \omega}(X))
\end{equation*}
where $\mathscr H^\pm_{\underline \omega}(X)$ are the spaces of harmonic (anti-)self-dual $2$-forms on $(X,\underline \omega)$ and $\R^*_3$ is the dual space of $\R^3$. Using this decomposition, we can describe the deformations of $\varphi_{\underline \omega}$ by analysing separately each component.
\begin{itemize}
    \item The first component $\Lambda^3 \R^*_3$ is spanned by $d\theta_1 \wedge d\theta_2 \wedge d\theta_3$. Deforming of $\varphi_{\underline \omega}$ along this direction corresponds to rescaling the inner product on $T^3$ by some factor $\lambda > 0$, together with a rescaling of the hyperk\"{a}hler triple $\underline \omega$ by a factor $\lambda^{-\frac{1}{2}}$.

    \item $\R^*_3 \otimes \mathscr H^+_{\underline \omega}(X)$ has dimension $9$ and contains $\mathscr H^3_{1}(T^3 \times X, \varphi_{\underline \omega})$ as a $3$-dimensional subspace spanned by $d\theta_j \wedge \omega_i - d\theta_i \wedge \omega_j = \frac{\partial}{\partial \theta_k} \lrcorner \Theta(\varphi_{\underline{\omega}})$ for cyclic permutations $(ijk)$ of $\{1,2,3\}$, corresponding to the isometric deformations of the $G_2$-structure $\varphi_{\underline \omega}$. Its orthogonal complement has dimension $6$, and decomposes as the direct sum of the $5$-dimensional space $\{\sum a_{ij} d\theta_i \wedge \omega_j, a_{ij} = a_{ji}, \sum a_{ii} = 0 \}$ corresponding to the infinitesimal deformations of the inner product on $T^3$ with fixed volume element, and a $1$-dimensional space spanned by $d \theta_1 \wedge \omega_1 + d\theta_2 \wedge \omega_2 + d\theta_3 \wedge \omega_3$ corresponding to an infinitesimal rescaling of the hyperk\"{a}hler triple.

    \item The third component $\R^*_3 \otimes \mathscr H^-_{\underline \omega}(X)$ corresponds to the deformations of the hyperk\"{a}hler metric $g_{\underline \omega}$ on $X$ with fixed volume, where the inner product on $\R^3$ is fixed.
\end{itemize}
Now assume that $\Gamma$ is a finite group acting by isometries on $T^3 \times X$, preserving $\varphi_{\underline \omega}$, and such that the quotient $M = (T^3 \times X) / \Gamma$ has $b^1(M) = 0$. We denote by $\varphi$ the torsion-free $G_2$-structure induced by $\varphi_{\underline{\omega}}$ on $M$. As the isometry group of $(T^3 \times X,g_{\varphi_{\underline \omega}})$ is isomorphic to the product of the isometry groups of $(T^3,d\theta^2_1 + d\theta_2^2+d\theta_3^2)$ and $(X,g_{\underline \omega})$, $\Gamma$ preserves the above decomposition of $\mathscr H^3(T^3 \times X, \varphi_{\underline \omega})$. The quotient map $\pi : T^3 \times X \rightarrow M$ induces an identification $\mathscr H^3(M, \varphi) \simeq \mathscr H^3(T^3 \times X, \varphi_{\underline \omega})^\Gamma$, and hence we obtain a decomposition:
\begin{equation}        \label{eq:decomph3gamma}
    \mathscr H^3(M, \varphi) = \Lambda^3 \R^*_3 \oplus  \left( \R^*_3 \otimes \mathscr H^+_{\underline \omega_0}(X) \right)^\Gamma \oplus \left( \R^*_3 \otimes \mathscr H^-_{\underline \omega_0}(X) \right)^\Gamma .
\end{equation}
Note that as $\varphi_{\underline \omega}$ is fixed by $\Gamma$, $d\theta_1 \wedge d\theta_2 \wedge d\theta_3$ must also be fixed by $\Gamma$. This decomposition induces a splitting $T \M = T^0 \M \oplus T^+\M \oplus T^-\M$ of the tangent bundle of $\M$. Using this splitting and Theorem \ref{thm:fabcd}, we can prove:

\begin{prop}
    Let $(M,\varphi)$ be a compact $G_2$-manifold with $b^1(M) = 0$ whose universal cover is $\R^3 \times K3$. Then the Yukawa coupling on $\M$ is parallel for the Levi-Civita connection of $\G$, and hence $(\M,\G)$ is locally symmetric.
\end{prop}

\begin{proof}
    Let us choose affine coordinates $(x^0,\ldots,x^n)$ near $\varphi \D$ in $\M$ and prove that the extra term $\mathscr{E}_{abcd} + \mathscr{E}_{cabd} + \mathscr{E}_{cbad}$ in Theorem \ref{thm:fabcd} vanishes. Let us write $n = n_+ + n_-$ where $n_{\pm}$ is the dimension of $( \R^*_3 \otimes \mathscr H^\pm_{\underline \omega_0}(X))^\Gamma$. Up to a linear change of coordinates, we can assume that we chose coordinates adapted to the decomposition \eqref{eq:decomph3gamma}, in the sense that the harmonic representative of $\frac{\partial}{\partial x^0} \in H^3(M)$ for the metric $g_\varphi$ lies in $\Lambda^3 \R^*_3$, the harmonic representatives of $\frac{\partial}{\partial x^1}, \ldots, \frac{\partial}{\partial x^{n_+}}$ lie in $(\R^*_3 \otimes \mathscr H^+_{\underline \omega_0}(X))^\Gamma$, and the harmonic representatives of $\frac{\partial}{\partial x^{n_++1}},\ldots,\frac{\partial}{\partial x^n}$ lie in $(\R^*_3 \otimes \mathscr H^-_{\underline \omega_0}(X))^\Gamma$. Note that this can only be imposed at the point $\varphi \D \in \M$, not locally near this point. Throughout the proof our computations will be local (in $M$), and therefore we can lift everything to $T^3 \times X$ using the quotient map $\pi$, where the variations of the space of harmonic forms are easier to understand (using the result of Lemma \ref{lem:hetadecomp}, which does not require the vanishing of the first Betti number but only all the harmonic forms to be of type $\Omega^3_1 \oplus \Omega^3_{27}$, on $T^3 \times X$).

    First we prove that if one of the indices $a,b,c$ or $d$ is between $0$ and $n_+$ then $\mathscr{E}_{abcd}(\varphi \D) = 0$, and similarly for $\mathscr{E}_{cabd}$ and $\mathscr{E}_{cbad}$. Since $\F_{abcd}$ is fully symmetric in its indices, we may assume that $0 \leq d \leq n_+$, and seek to prove that $h_d \cdot \eta$ is harmonic for any $\eta \in \mathscr{H}^3(M,\varphi)$. As a consequence of our discussion of the deformations of $\varphi_{\underline \omega}$ on $T^3 \times X$, there is a deformation $\{\varphi_{\underline{\omega}_t}\}_{t \in (-\epsilon,\epsilon)}$ of $\varphi_{\underline{\omega}}$ on $T^3 \times X$ which consists in a variation of the inner product on $T^3$ combined with a rotation and a dilation of the hyperk\"ahler triple on $X$, and such that $\left. \frac{\partial \varphi_{\underline{\omega}_t}}{\partial t} \right|_{t=0}$ is the lift of $\eta_d$. In particular the space of harmonic forms on $T^3 \times X$ with respect to $g_{\varphi_{\underline{\omega}_t}}$ is fixed along this deformation of $\varphi_{\underline{\omega}}$. Hence Lemma \ref{lem:hetadecomp} implies that the lift of $h_d \cdot \eta$ to $T^3 \times X$ is harmonic whenever $\eta$ is a harmonic form on $M$, and thus $h_d \cdot \eta$ is harmonic on $M$. Hence $\mathscr{E}_{abcd}(\varphi \D) = \mathscr{E}_{cabd}(\varphi \D) = \mathscr{E}_{cbad}(\varphi \D) = 0$.

    Now let us assume that $n_+ +1 \leq a,b,c,d \leq n$. In this case, it is no longer true that $h_d \cdot \eta_c$ is harmonic, but we want to prove that the contribution to $\mathscr{E}_{abcd}$ of its exact part cancels with the contribution of the co-exact part. This time, our discussion of the deformations of torsion-free $G_2$-structures on $T^3 \times X$ implies that there is a deformation $\{\varphi_{\underline{\omega}_t}\}_{t \in (-\epsilon,\epsilon)}$ of $\varphi_{\underline{\omega}}$ such that $\left. \frac{\partial \varphi_{\underline{\omega}_t}}{\partial t} \right|_{t=0} = \pi^* \eta_d$ is the lift of $\eta_d$ and $\varphi_{\underline{\omega}_t}$ can be written
    \begin{equation*}
        \varphi_{\underline{\omega}_t} = d\theta_1 \wedge d\theta_2 \wedge d\theta_3 - \sum_{j=1}^3 d\theta_j \wedge \omega_{j,t}
    \end{equation*}
    where $\underline \omega_t = (\omega_{1,t},\omega_{2,t},\omega_{3,t})$ is a family of hyperk\"ahler triples on $X$. Now let $\eta \in \mathscr{H}^3(M,\varphi)$, representing a vector in $T^-_\varphi \M$. Its lift $\widetilde \eta = \pi^* \eta$ on $T^3 \times X$ can be written
    \begin{equation*}
        \widetilde \eta = d\theta_1 \wedge \alpha_1 + d\theta_2 \wedge \alpha_2 + d\theta_3 \wedge \alpha_3
    \end{equation*}
    where $\alpha_1,\alpha_2,\alpha_3$ are anti-self-dual harmonic $2$-forms on $X,g_{\underline \omega}$. In particular the dual $4$-form of $\widetilde \eta$, which we denote by $\widetilde \nu = *_{\varphi_{\underline \omega}} \widetilde \eta = \pi^*(*_\varphi \eta)$ is
    \begin{equation*}
        \widetilde \nu = - d\theta_2 \wedge d\theta_3 \wedge \alpha_1 - d\theta_3 \wedge d\theta_1 \wedge \alpha_2 - d\theta_1 \wedge d\theta_2 \wedge \alpha_3 .
    \end{equation*}
    If we now denote by $\widetilde \eta_t$ the harmonic representative of $[\widetilde \eta] = \pi^*[\eta] \in H^3(T^3 \times X)$ for the metric $g_{\varphi_{\underline{\omega}_t}}$ and $\widetilde \nu_t$ the harmonic representative of $[\widetilde \nu] = \pi^*[*_\varphi \eta] \in H^4(T^3 \times X)$, we see that
    \begin{align*}
        \widetilde \eta_t & = d\theta_1 \wedge \alpha_{1,t} + d\theta_2 \wedge \alpha_{2,t} + d\theta_3 \wedge \alpha_{3,t} , \\
        \widetilde \nu_t & =  - d\theta_2 \wedge d\theta_3 \wedge \alpha_{1,t} - d\theta_3 \wedge d\theta_1 \wedge \alpha_{2,t} - d\theta_1 \wedge d\theta_2 \wedge \alpha_{3,t}
    \end{align*}
    where $\alpha_{j,t}$ is the harmonic representative of $[\alpha_j] \in H^2(X)$ for the hyperk\"ahler metric associated with $\underline \omega_t$. In particular, the lift of the exact part of $2(h_d \cdot \eta)$ to $T^3 \times X$ is
    \begin{equation*}
        \left. \frac{\partial \widetilde \eta_t}{\partial t} \right|_{t=0} = d\theta_1 \wedge \left. \frac{\partial \alpha_{1,t}}{\partial t} \right|_{t=0} + d\theta_2 \wedge \left. \frac{\alpha_{2,t}}{\partial t} \right|_{t=0} + d\theta_3 \wedge \left. \frac{\alpha_{3,t}}{\partial t} \right|_{t=0}
    \end{equation*}
    and the lift of its co-exact part is
    \begin{equation*}
        \left. \frac{\partial \widetilde \nu_t}{\partial t} \right|_{t=0} = - d\theta_2 \wedge d\theta_3 \wedge \left. \frac{\partial \alpha_{1,t}}{\partial t} \right|_{t=0} - d\theta_3 \wedge d\theta_1 \wedge \left. \frac{\alpha_{2,t}}{\partial t} \right|_{t=0} - d\theta_1 \wedge d\theta_2  \wedge \left. \frac{\alpha_{3,t}}{\partial t} \right|_{t=0} \cdot
    \end{equation*}
    If we now let $\eta = \eta_c$ and describe in a similar way the exact and co-exact parts of $h_a \cdot \eta_b$, we see that the inner product of the exact parts of $h_d \cdot \eta_c$ and $h_a \cdot \eta_b$ is equal to the inner product of their co-exact parts, and thus $\mathscr{E}_{abcd}(\varphi \D) = 0$ (see Remark \ref{rem:exactparts}). Similarly $\mathscr{E}_{cabd}(\varphi \D) = \mathscr{E}_{cbad}(\varphi \D) = 0$, which completes the proof of the proposition.
\end{proof}



    \section{A period mapping}        \label{section:periods}

In this section, we introduce an immersion $\Phi$ of the moduli space $\M$ into the homogeneous space $GL(n+1) / (\{\pm 1\} \times O(n))$, and show that it naturally determines the geometric structures of $\M$. The idea is inspired by the period map introduced by Griffiths on Calabi--Yau moduli spaces \cite{griffiths1968periodsi,griffiths1968periodsii}, and the related notion of Weil--Petersson geometry of Lu and Sun \cite{lu2004weil}.

By means of motivation, let us recall a few facts. If $Y$ is a compact Calabi--Yau threefold, the cohomology group $H^3(Y;\C)$ admits a Hodge decomposition 
\begin{equation*}
    H^3(Y;\C) = H^{3,0} \oplus H^{2,1} \oplus H^{1,2} \oplus H^{0,3} .
\end{equation*}
The cup-product induces a symplectic structure $Q$ on $H^3(Y;\C)$ and the Hodge decomposition is subject to the following conditions:
\begin{enumerate}
    \item[\textbf{(A)}] $\overline{H^{p,3-p}} = H^{3-p,p}$, for all $p = 0,1,2,3$.

    \item[\textbf{(B)}] $iQ(H^{p,3-p},\overline{H^{q,3-q}}) = 0$ if $p \neq q$, and $(-1)^{p+1}iQ(H^{p,3-p},\overline{H^{p,3-p}}) > 0$ for all $p$.

    \item[\textbf{(C)}] $\dim H^{3,0} = 1$.
\end{enumerate}
By considering the Hodge filtration $F^p = H^{3,0} \oplus \cdots \oplus H^{3-p,p}$, it can be shown that the domain parametrising  such decompositions (called Hodge structures of weight $(1,h^{2,1})$) is a complex homogeneous space diffeomorphic to $Sp(Q) / (U(1) \times U(n))$, where $Sp(Q)$ is the group of real endomorphisms of $H$ preserving the symplectic form $Q$ and $n = h^{2,1}(Y)$. Griffiths proved that the Hodge filtration varies holomorphically along an analytic deformation of the complex structure of $Y$, and that these variations satisfy the \emph{transversality condition} $dF^p \subset F^{p-1}$ \cite{griffiths1968periodsi,griffiths1968periodsii}. This condition in particular implies that the Weil--Peterson metric can be seen as the pull-back of a homogeneous \emph{indefinite} hermitian form defined on the period domain \cite{tian1987smoothness,todorov1989weil,lu2004weil}. 

This section is organised as follows. The map $\Phi$ is defined in \S\ref{subsection:firstdef}, where we also describe the structure of the target domains. In \S\ref{subsection:infinitesimal}, we describe the properties of $\Phi$ and its relation to the metric $\G$. In \S\ref{subsection:condition}, we show that the extra terms in Theorem \ref{thm:fabcd} essentially correspond to the second fundamental form of $\Phi$. Finally, in \S\ref{subsection:comments} we relate the map $\Phi$ to the usual notion of period map for $G_2$-manifolds, as a Lagrangian immersion of $\M$ into $H^3(M) \oplus H^4(M)$.

    \subsection{First definitions}      \label{subsection:firstdef}

Let $(M,\varphi)$ be a compact $G_2$-manifold, with the usual assumption $b^1(M) = 0$. For simplicity, we will write $H^3 = H^3(M)$, $H^4 = H^4(M)$ and $H = H^3 \oplus H^4$, and $n = b^3_{27}(M) = b^3(M)-1$. We can define an involution $\iota = \Id_{H^3} - \Id_{H^4}$ on $H$. As the cup-product identifies $H^4$ with the dual space of $H^3$, $H$ is endowed with a natural symplectic form $Q$. Explicitly, if $\eta,\eta^\prime$ are closed $3$-forms and $\nu,\nu^\prime$ closed $4$-forms we have
\begin{equation*}
    Q([\eta]+[\nu],[\eta^\prime]+[\nu^\prime]) = \int_M \eta \wedge \nu^\prime - \int_M \eta^\prime \wedge \nu .
\end{equation*}
Let us consider the decomposition $H = H^{(3)}_\varphi \oplus H^{(2)}_\varphi \oplus H^{(1)}_\varphi \oplus H^{(0)}_\varphi$ defined by
\begin{align}
    \begin{split}
        H^{(3)}_\varphi & = \{ [\eta] + [*_\varphi \eta], ~ \eta \in \mathscr H^3_{1}(M,\varphi) \}, \\
        H^{(2)}_\varphi & = \{ [\eta] - [*_\varphi \eta], ~ \eta \in \mathscr H^3_{27}(M,\varphi) \}, \\
        H^{(1)}_\varphi & = \{ [\eta] + [*_\varphi \eta], ~ \eta \in \mathscr H^3_{27}(M,\varphi) \}, ~ \text{and}\\
        H^{(0)}_\varphi & = \{ [\eta] - [*_\varphi \eta], ~ \eta \in \mathscr H^3_{1}(M,\varphi) \} .
    \end{split}
\end{align}
It satisfies the following properties, analogous to \textbf{(A)}, \textbf{(B)} and \textbf{(C)} above:
\begin{enumerate}
    \item[\textbf{(1)}] $H^{(3-p)}_\varphi = \iota(H^{(p)}_\varphi)$ for $p= 0, \ldots, 3$.
    \item[\textbf{(2)}] $Q(\iota(H^{(p)}_\varphi), H^{(q)}_\varphi) = 0$ if $p \neq q$, and $(-1)^{p+1} Q(\iota(H^{(p)}_\varphi),H^{(p)}_\varphi) > 0$, for any $0 \leq p,q \leq 3$.
    \item[\textbf{(3)}] $\dim H^{(3)}_\varphi = 1$ and $\dim H^{(2)}_\varphi = n$.
\end{enumerate}
Let us denote by $\Dom \subset \mathbb{P}(H) \times Gr(n,H) \times Gr(n,H) \times \mathbb{P}(H)$ the set of decompositions $\mathbf{H} = (H^{(3)},H^{(2)},H^{(1)},H^{(0)})$ of $H$ satisfying the above properties. The subgroup of $GL(H)$ of automorphisms fixing $(Q,\iota)$ naturally acts on $\Dom$. This group can be identified with $GL(H^3) \simeq GL(n+1)$. Explicitly, if we fix bases $(u_0,\ldots,u_n)$ of $H^3$ and $(v_0,\ldots,v_n)$ of $H^4$ such that
\begin{equation}        \label{eq:standardbasis}
    Q(u_i,v_j) = \delta_{ij}, ~~~~ \forall 0 \leq i,j \leq n
\end{equation}
then any matrix $A \in GL(n+1)$ acts on $H$ via
\begin{equation}
    A(u_j) = \sum_{i=0}^N A_{ij} u_i , ~~~ A(v_j) = \sum_{i=0}^N (A^{-1})_{ji} v_i .
\end{equation}
This action has the following properties:

\begin{lem}
    There is an equivariant diffeomorphism $\Dom \rightarrow \mathbb{P}(H^3) \times S^2_+ (H^3)^*$. In particular the action of $GL(H^3)$ on $\Dom$ is transitive, and $\Dom \simeq GL(n+1) / (\{\pm 1\} \times O(n))$.
\end{lem}

\begin{proof}
    If $\mathbf{H} \in \Dom$, we can define a line $\ell_{\mathbf{H}} \in \mathbb{P}(H^3)$ by
    \begin{equation*}
        \ell_{\mathbf{H}} = \{ w + \iota(w), ~ w \in H^{(3)} \} .
    \end{equation*}
    There is also a quadratic form $q_{\mathbf{H}}$ on $H^3$ defined as
    \begin{equation*}
        q_{\mathbf{H}}(u) = 2 Q(\pi^{(0)}_{\mathbf{H}}u, \pi^{(3)}_{\mathbf{H}} u) - 2 Q(\pi^{(1)}_{\mathbf{H}} u, \pi^{(2)}_{\mathbf{H}} u), ~~~ \forall u \in H^3,
    \end{equation*}
    where $\pi^{(p)}_{\mathbf{H}}$ denotes the projection of $H$ onto $H^{(p)}$ in the decomposition $H = H^{(3)} \oplus H^{(2)} \oplus H^{(1)} \oplus H^{(0)}$. Properties \textbf{(1)} and \textbf{(2)} imply that $q_{\mathbf{H}}$ is positive definite on $H^3$, and thus $q_{\mathbf{H}} \in S^2_+ (H^3)^*$. This way we have defined a map $\Dom \rightarrow \mathbb{P}(H^3) \times S^2_+(H^3)^*$, and it is clear that it is equivariant under the action of $GL(H^3)$. This map is invertible, and its inverse can be constructed as follows. Let $(\ell,q) \in \mathbb{P}(H^3) \times S^2_+ (H^3)^*$, and let $(u_0,\ldots,u_n)$ be an orthonormal basis of $H^3$ such that $u_0$ spans $\ell$. Then there exists a unique basis $(v_0,\ldots,v_n)$ of $H^4$ such that $Q(u_i,v_j) = \delta_{ij}$, and we can define
    \begin{equation*}
        H^{(3)}_{(\ell,q)} = \Span \{ u_0 + v_0 \}, ~~~ H^{(2)}_{(\ell,q)} = \Span \{ u_j - v_j, 1 \leq j \leq n \}
    \end{equation*}
    as well as $H^{(p)}_{(\ell,q)} = \iota(H^{(3-p)}_{(\ell,q)})$ for $p = 0,1$. It is easy to see that this decomposition $\mathbf{H}_{(\ell,q)}$ is an element of $\Dom$, and that the map $\mathbb{P}(H^3) \times S^2_+ (H^3)^*$ defined in this way is an inverse for the map $\mathbf{H} \mapsto (\ell_{\mathbf{H}},q_{\mathbf{H}})$. The rest of the lemma follows.
\end{proof}

\begin{rem}
    Under the diffeomorphism $\Dom \simeq \mathbb{P}(H^3) \times S^2_+ (H^3)^*$, we can easily see that for any torsion-free $G_2$-structure $\varphi$ on $M$ we have $\ell(\flag_\varphi) = H^3_1(M,\varphi)$, and $q(\flag_\varphi)$ is the inner product on $H^3$ induced by the $L^2$-inner product on $\mathscr{H}^3(M,g_\varphi)$.
\end{rem}

Throughout this section, it will be convenient to adopt the following definition. If $\flag \in \Dom$, a basis $(u_0,\ldots,u_n,v_0,\ldots,v_n)$ of $H$ will be called a \emph{standard basis} for $\flag$ if it satisfies the following properties:
\begin{enumerate}[(i)]
    \item $(u_0,\ldots,u_n)$ is a basis of $H^3$, $(v_0,\ldots,v_n)$ is a basis of $H^4$, and relations \eqref{eq:standardbasis} are satisfied.
    \item The basis $(u_0,\ldots,u_n)$ is orthonormal for the inner product $q_\flag$.
    \item $H^{(3)} = \Span \{ u_0 + v_0 \}$ and $H^{(2)} = \Span \{ u_i - v_i, ~ 1 \leq i \leq n \}$.
\end{enumerate}
Standard bases always exist, and are uniquely determined by a $q_\flag$-orthonormal basis $(u_0,\ldots,u_n)$ of $H^3$ such that $u_0 \in \ell_\flag$.

Let us denote by $G_\flag \subset GL(H^3)$ the stabiliser of an element $\flag \in \Dom$, and by $\mathfrak{g}_\flag \subset \mathfrak{gl}(H^3)$ its Lie algebra. In a standard basis $(u_0,\ldots,u_n,v_0,\ldots,v_n)$ of $H$ associated with $\flag$, $\mathfrak{g}_\flag$ corresponds to the space of matrices 
\begin{equation*}
    \mathfrak{g}_\flag = \{ (a_{ij})_{0 \leq i,j \leq n}, ~ a_{0i} = a_{i0} = 0 ~ \forall 0 \leq i \leq n, ~ a_{ij} = - a_{ij} ~ \forall 1 \leq i,j \leq n \} .
\end{equation*}
The quadratic form $q_\flag$ induces an inner product on $\mathfrak{gl}(H^3)$: if $a \in \mathfrak{gl}(H^3)$ corresponds to the matrix $(a_{ij})_{0 \leq i,j \leq n}$ in the basis $(u_0,\ldots,u_n)$, we have:
\begin{equation*}
    |a|^2_\flag = \sum_{i,j=0}^n a_{ij}^2.
\end{equation*}
We denote by $\mathfrak{p}_\flag$ the orthogonal complement of $\mathfrak{g}_\flag$ for this inner product. That is, in a standard basis,
\begin{equation*}
    \mathfrak{p}_\flag = \{ (a_{ij})_{0 \leq i,j \leq n}, ~ a_{ij} = a_{ij} ~ \forall 1 \leq i,j \leq n \} .
\end{equation*}
The tangent space $T_\flag \Dom$ can be identified with $\mathfrak{p}_\flag$, which is endowed with the inner product induced by $q_\flag$. This defines a Riemannian metric $g_\Dom$ on $\Dom$, homogeneous with respect to $GL(H^3)$. Let us denote by $T^v \Dom$ the distribution tangent to the fibres of $q : \Dom \rightarrow S^2_+(H^3)^*$ and call it the \emph{vertical distribution} of $\Dom$. The \emph{horizontal distribution} of $\Dom$ is defined as the orthogonal complement of the vertical distribution, and will be denoted by $T^h \Dom$. If $\flag \in \Dom$ and $T_\flag \Dom$ is identified with $\mathfrak{p}_\flag \subset \mathfrak{gl}(H^3)$, then the splitting $T_\flag \Dom = T^v_\flag \Dom \oplus T^h_\flag \Dom$ corresponds to the decomposition $\mathfrak{p}_\flag = \mathfrak{v}_\flag \oplus \mathfrak{h}_\flag$, where $\mathfrak{h}_\flag$ is the space of endomorphisms of $H^3$ that are self-adjoint with respect to the inner product $q_\flag$ and $\mathfrak{v}_\flag$ its orthogonal complement in $\mathfrak{p}_\flag$. In particular, the map $q : \Dom \rightarrow S^2_+ (H^3)^*$ is a Riemannian fibration for the natural symmetric metric $g_{S^2_+}$ of $S^2_+(H^3)^*$. Recall that this metric can be defined as follows: if $q \in S^2_+(H^3)^*$ is an inner product and $\dot q \in S^2(H^3)^* \simeq T_q S^2_+(H^3)^*$, there is a unique $q$-self-adjoint endomorphism $a$ of $H^3$ such that $\dot q = \left. \frac{d}{dt} \right|_{t=0} (e^{ta})^*q = q(a \cdot,\cdot) + q(\cdot, a \cdot) = 2 q(a\cdot,\cdot)$, and then we define $|\dot q|_q^2 = \tr(a^2) = \sum a_{ij}^2$ in a $q$-orthonormal basis.

Written in a standard basis, the horizontal and vertical spaces are given by
\begin{align*}
    \mathfrak{v}_\flag & = \{ (a_{ij})_{0 \leq i,j \leq n}, ~ a_{0i} = - a_{i0} ~ \forall 0 \leq i \leq n, ~ a_{ij} = 0 ~ \forall 1 \leq i,j \leq n\}, \\
    \mathfrak{h}_\flag & = \{ (a_{ij})_{0 \leq i,j \leq n}, ~ a_{ij} = a_{ij} ~ \forall 0 \leq i,j \leq n\} .
\end{align*}
The horizontal distribution admits a further equivariant splitting. By the previous lemma, $\flag$ determines a line $\ell_\flag \in \mathbb{P}(H^3)$ which is fixed by $G_\flag$, and therefore there is a $1$-dimensional subspace $\mathfrak{l}_\flag \subset \mathfrak{h}_\flag$ consisting of those self-adjoint endomorphisms that send $\ell_\flag$ to itself and act trivially on its orthogonal complement. We denote by $\mathfrak{t}_\flag$ the orthogonal complement of $\mathfrak{l}_\flag$ in $\mathfrak{h}_\flag$ and by $T^t_\flag \Dom$ the corresponding subspace of $T_\flag \Dom$. This defines an equivariant distribution $T^t \Dom \subset T \Dom$, which we call the \emph{transverse distribution} of $\Dom$. Again, in a standard basis we have
\begin{align*}
    \mathfrak{l}_\flag & = \{ (a_{ij})_{0 \leq i,j \leq n}, ~ a_{ij} = 0 ~ \text{if} ~ (i,j) \neq (0,0)\}, \\
    \mathfrak{t}_\flag & = \{ (a_{ij})_{0 \leq i,j \leq n}, ~ a_{ij} = a_{ij} ~ \forall 0 \leq i,j \leq N, ~ a_{00} = 0\}.
\end{align*}

Another convenient description of the horizontal and transverse distributions can be given by introducing the filtration $F^{(3)} \subset F^{(2)} \subset F^{(1)} \subset F^{(0)} = H$ associated with $\flag \in \Dom$:
\begin{equation*}
    F^{(p)} = H^{(3)} \oplus \cdots \oplus H^{(p)} .
\end{equation*}
Clearly this filtration determines $\flag$, and therefore this defines an equivariant embedding of $\Dom$ in a manifold of flags in $H$. Via this embedding, any tangent vector $\xi \in T_\flag \Dom$ can be represented by a triple of linear maps $F^{(p)} \rightarrow H / F^{(p)}$ for $p=1,2,3$. Since $F^{(p)} \subset F^{(p-1)}$ and $H^{(p-1)} \oplus \cdots \oplus H^{(0)}$ is a complement of $F^{(p)}$, we can in fact represent $\xi$ by $(\phi^{(3)}_\xi,\phi^{(2)}_\xi,\phi^{(1)}_\xi)$ where
\begin{equation*}
    \phi^{(p)}_\xi : H^{(p)} \rightarrow H^{(p-1)} \oplus \cdots \oplus H^{(0)} .
\end{equation*}

\begin{lem}     \label{lem:hortangent}
    Let $\flag \in \Dom$ and $\xi \in T_\flag \Dom$ be represented by the triple of linear maps $(\phi^{(3)}_\xi,\phi^{(2)}_\xi,\phi^{(1)}_\xi)$. Then $\xi$ is a horizontal vector if and only if 
    \begin{equation*}
        \phi^{(3)}_\xi(H^{(3)}) \subseteq H^{(2)} \oplus H^{(0)}
    \end{equation*}
    and in this case
    \begin{equation*}
        \phi^{(2)}_\xi(H^{(2)}) \subseteq H^{(1)} .
    \end{equation*}
    Moreover, $\xi$ is transverse if and only if 
    \begin{equation*}
        \phi^{(3)}_\xi(H^{(3)}) \subseteq H^{(2)}.
    \end{equation*}
    In particular if $\xi$ is transverse then $\phi^{(p)}_\xi \in \Hom(H^{(p)},H^{(p-1)})$.
\end{lem}

\begin{proof}
    Let $(u_0,\ldots,u_n,v_0,\ldots,v_n)$ be a standard basis of $H$ associated with $\flag$. In this basis, $\Dom \simeq GL(n+1) / (\{ \pm 1\} \times O(n))$ and the vector $\xi \in T_\flag \Dom$ is uniquely represented by a matrix $a_\xi=(a_{ij})_{0\leq i,j \leq n}$ satisfying
    \begin{equation*}
        a_{ji} = a_{ij} , ~~~\forall 1 \leq i, j \leq n .
    \end{equation*}
    Now $a_\xi$ acts on $H^3$ by $a(u_j) = a_{ij} u_i$ and on $H^4$ by $a(v_j) = - a_{ji} v_i$, and therefore the linear map $\phi^{(3)}_\xi$ is characterised by:
    \begin{align*}
        \phi^{(3)}_\xi(u_0 + v_0) & = \sum_{i=0}^n a_{i0}u_i - a_{0i} v_i \\
        & = a_{00}(u_0 - v_0) + \sum_{i=1}^n ~ \left \{ \frac{a_{i0} + a_{0i}}{2}(u_i-v_i) + \frac{a_{i0}-a_{0i}}{2}(u_i+v_i) \right \} ~,
    \end{align*}
    where the first term belongs to $H^{(0)}$, the second term to $H^{(2)}$ and the third to $H^{(1)}$. Hence $\phi^{(3)}_\xi$ maps into $H^{(2)} \oplus H^{(0)}$ if and only if $a_{0i} = a_{i0}$, that is if $a_\xi$ is symmetric. This is exactly the condition for $\xi$ to define a horizontal vector in $T_\flag \Dom$. Moreover $\phi^{(3)}_\xi$ maps into $H^{(2)}$ if and only if $a_\xi$ is symmetric and $a_{00} = 0$, that is, if $a_\xi \in \mathfrak{t}_\flag$.

    Now assume that $\xi$ is a horizontal vector, that is, $a_\xi$ is symmetric. The only nontrivial inclusion left to check is $\phi^{(2)}_\xi(H^{(2)}) \subset H^{(1)}$. On $H^{(2)}$, $a_\xi$ acts by
    \begin{equation*}
        a_\xi(u_j-v_j) = \sum_{i=0}^n a_{ij} u_i + a_{ji} v_i = a_{0j}(u_0+v_0) + \sum_{i=1}^n a_{ij} (u_i+v_i) 
    \end{equation*}
    where the first term $a_{0j}(u_0 + v_0) \in H^{(3)}$ and the second term is an element of $H^{(1)}$. Only the projection of $a_\xi(u_j-v_j)$ onto $H^{(1)} \oplus H^{(0)}$ contributes to $\phi^{(2)}_\xi(u_j-v_j)$ and therefore $\phi^{(2)}_\xi(H^{(2)}) \subseteq H^{(1)}$.
\end{proof}

    \subsection{Infinitesimal variations and Riemannian aspects}        \label{subsection:infinitesimal}

As the decomposition $H = H^{(3)}_\varphi \oplus H^{(2)}_\varphi \oplus H^{(1)}_\varphi \oplus H^{(0)}_\varphi$ associated with a torsion-free $G_2$-structure $\varphi$ only depends on the class of $\varphi$ modulo $\D$, there is a well-defined map $\Phi : \M \rightarrow \Dom$. This is a smooth map, and it has the following properties, which are analogous to the properties of the period map on the moduli spaces of Calabi--Yau threefolds:

\begin{thm}        \label{thm:horizontal}
    The map $\Phi : \M \rightarrow \Dom$ is a horizontal immersion, and the restriction of $\Phi$ to $\M_1$ is transverse. 
    
    Moreover, if $\varphi \in \M_1$ and $\eta \in \mathscr H^3_{27}(M,\varphi) \simeq T_{\varphi \D} \M_1$, then $T_{\varphi \D} \Phi(\eta)$ is determined by the triple of linear maps $\phi^{(p)}_\eta \in \Hom(H^{(p)}_\varphi,H^{(p-1)}_\varphi)$, $p=1,2,3$, defined as follows. Let $h$ be the unique trace-free self-adjoint endomorphism such that $h \cdot \varphi = \eta$ and let $\eta^\prime \in \mathscr H^3_{27}(M,\varphi)$. Then we have:
    \begin{enumerate}[(i)]
        \item $\phi^{(3)}_\eta([\varphi] + [\Theta(\varphi)]) = [\eta] - [*_\varphi \eta]$, 
        
        \item $\phi^{(2)}_\eta([\eta^\prime]-[*\eta^\prime]) = [\pi_{27}\mathscr H(h \cdot \eta^\prime)] + [*_\varphi \pi_{27}\mathscr{H}(h \cdot \eta^\prime)) ]$,

        \item $\phi^{(1)}_\eta([\eta^\prime]+[*\eta^\prime]) = \frac{1}{7} \int \langle \eta^\prime, \eta \rangle_{\varphi} \mu_\varphi \cdot ([\varphi] - [\Theta(\varphi)])$.
    \end{enumerate}
\end{thm}

\begin{proof}
    Let $\{\varphi_t\}_{t \in (-\epsilon,\epsilon)}$ be a family of torsion-free $G_2$-structures on $M$ such that $\varphi_0 = \varphi$ and assume that $\left. \frac{\partial \varphi_t}{\partial t} \right|_{t = 0} = \eta$ is a harmonic $3$-form. Let $\flag_t = \Phi(\varphi_t)$ and $(\phi^{(3)}_\eta,\phi^{(2)}_\eta,\phi^{(1)}_\eta)$ be the triple of linear map representing $T_{\varphi \D} \Phi(\eta)$. For all $t \in (-\epsilon,\epsilon)$, $H^{(3)}_t \subset H$ is spanned by $[\varphi_t] + [\Theta(\varphi_t)]$. Differentiating at $t=0$ we have
    \begin{align}   \label{eq:diffphi3}
        \begin{split}
            \left. \frac{\partial \varphi_t}{\partial t} \right|_{t=0} + \left. \frac{\partial \Theta(\varphi_t)}{\partial t} \right|_{t=0} & = \eta + \frac{4}{3} *_\varphi \pi_1(\eta) - *_\varphi \pi_{27}(\eta) \\
                & = \pi_1(\eta) + \frac{4}{3} *_\varphi \pi_1(\eta) + \pi_{27}(\eta) - *_\varphi \pi_{27}(\eta).
        \end{split}
    \end{align}
    Since $\eta$ is harmonic with respect to $g_\varphi$, the first two terms term represent an element of $H^{(3)}_\varphi \oplus H^{(0)}_\varphi$, and the last two terms an element of $H^{(2)}_\varphi$, and hence $\phi^{(3)}_\eta(H^{(3)}_\varphi) \subseteq H^{(2)}_\varphi \oplus H^{(0)}_\varphi$. If moreover all $\varphi_t$ have unit volume then $\pi_1(\eta) = 0$, and thus $\phi^{(3)}_\eta(H^{(3)}_\varphi) \subset H^{(2)}_\varphi$. Hence the first part of the theorem follows from the previous lemma.

    Let us now assume that $\Vol(\varphi_t) = 1$ for all $t$ and let us compute the differential of $\Phi$. The expression for $\phi^{(3)}_\eta$ follows from \eqref{eq:diffphi3} since the first two terms vanish. Now let $\eta_1,\ldots,\eta_n$ be a basis of $\mathscr{H}^3_{27}(M,\varphi)$, and denote by $\eta_{a,t}$ the element of $\mathscr{H}^3(M,\varphi_t)$ such that $[\eta_{a,t}] = [\eta_a] \in H^3(M)$. For small enough $t$, the differential forms $\eta^\prime_{a,t}$ defined by
    \begin{equation*}
        \eta^\prime_{a,t} = \eta_{a,t} - \frac{1}{7} \int \left( \eta_{a,t} \wedge \Theta(\varphi_t) \right) \cdot \varphi_t
    \end{equation*}
    form a basis of $\mathscr{H}^3_{27}(M,\varphi_t)$. Thus $H^{(2)}_t$ is spanned by the cohomology classes $[\eta^\prime_{a,t}]-[*_t \eta^\prime_{a,t}]$, $a = 1,\ldots,n$, for small $t$. At $t = 0$, each $\eta_a$ is orthogonal to $\varphi$ and thus by Lemma \ref{lem:firstvariationphi} we obtain
    \begin{equation*}
        \left. \frac{\partial \eta^\prime_{a,t}}{\partial t} \right|_{t=0} = \left. \frac{\partial \eta_{a,t}}{\partial t} \right|_{t=0} + \left( \frac{1}{7} \int \langle \eta_{a}, \eta \rangle_\varphi \mu_\varphi \right) \cdot \varphi
    \end{equation*}
    where $\eta = \left. \frac{\partial \varphi_t}{\partial t} \right|_{t=0}$. In particular since $\left. \frac{\partial \eta_{a,t}}{\partial t} \right|_{t=0}$ is exact the harmonic part of $\left. \frac{\partial \eta^\prime_{a,t}}{\partial t} \right|_{t=0}$ is $(\frac{1}{7} \int \langle \eta_a, \eta \rangle_\varphi \mu_\varphi) \varphi$. On the other hand, if we write $\eta = h \cdot \varphi$ where $h$ is traceless and self-adjoint, then by Lemma \ref{lem:firstvariationh} and Corollary \ref{cor:hstar} we have
    \begin{equation*}
        \left. \frac{\partial *_t \eta^\prime_{a,t}}{\partial t} \right|_{t=0} = h \cdot *_\varphi \eta_a - *_\varphi (h\cdot \eta_a) + *_\varphi \left. \frac{\partial \eta^\prime_{a,t}}{\partial t} \right|_{t=0}
    \end{equation*}
    and since $h$ anticommutes with $*_\varphi$, the harmonic part of $\left. \frac{\partial *_t \eta^\prime_{a,t}}{\partial t} \right|_{t=0}$ is
    \begin{equation*}
        - 2 *_\varphi \mathscr{H}(h \cdot \eta_a) + \left( \frac{1}{7} \int \langle \eta_a, \eta \rangle_\varphi \mu_\varphi \right) \cdot \Theta(\varphi).
    \end{equation*}
    Moreover, we have
    \begin{equation*}
        \pi_1\mathscr{H}(h \cdot \eta_a) = \left( \frac{1}{7} \int \langle h \cdot \eta, \varphi \rangle_\varphi \mu_\varphi \right) \cdot \varphi = \left( \frac{1}{7} \int \langle \eta_{a}, \eta \rangle_\varphi \mu_\varphi \right) \cdot \varphi
    \end{equation*}
    and thus gathering all the results we obtain
    \begin{align*}
        \left. \frac{\partial ([\eta^\prime_{a,t}] - [*_t \eta^\prime_{a,t}])}{\partial t} \right|_{t=0} &= 2 [*_\varphi \pi_{27}(h \cdot \eta_a)] + \left( \frac{1}{7} \int \langle \eta, \eta_a \rangle_\varphi \mu_\varphi \right) \cdot ([\varphi] + [\Theta(\varphi)])  \\
            & \equiv [\pi_{27}\mathscr{H}(h \cdot \eta_a)] + [*_\varphi \pi_{27}\mathscr{H}(h \cdot \eta_a)] \mod F^{(2)}_t.
    \end{align*}
    This yields the claimed expression for $\phi^{(2)}_\eta$. By a mere change of sign, the expression for $\phi^{(1)}_\eta$ follows from the fact that
    \begin{equation*}
        \left. \frac{\partial ([\eta^\prime_{a,t}] - [*_t \eta^\prime_{a,t}])}{\partial t} \right|_{t=0} \equiv \left( \frac{1}{7} \int \langle \eta, \eta_a \rangle_\varphi \mu_\varphi \right) \cdot ([\varphi] - [\Theta(\varphi)]) \mod F^{(1)}_t .
    \end{equation*}
    This finishes the proof of the theorem.
\end{proof}

The map $\Phi : \M \rightarrow \Dom$ is not a local isometry for the metrics $\G$ on $\M$ and $g_\Dom$ on $\Dom$. Nonetheless, it naturally determines the metric $\G$. Since $\G = 7 dt^2 + \G_1$ under the splitting $\M \simeq \R \times \M_1$, it is enough to prove that the restriction of $\Phi$ to $\M_1$ determines the metric $\G_1$. Because the map $\Phi : \M_1 \rightarrow \Dom$ is transverse, it turns out that $\G_1$ can be seen as the pull-back of an \emph{indefinite} quadratic form $h_\Dom$ on the transverse distribution. To define $h_\Dom$, let $\flag \in \Dom$, and consider a transverse tangent vector $\xi \in T_\flag^t \Dom$. By Lemma \ref{lem:hortangent}, it can be represented by a triple of linear maps $(\phi^{(3)}_\xi,\phi^{(2)}_\xi,\phi^{(1)}_\xi)$ where $\phi^{(p)}_\xi \in \Hom(H^{(p)},H^{(p-1)})$. If $w \in H^{(3)} \backslash \{0\}$, define
\begin{equation*}
    h_\Dom(\xi,\xi) = - \frac{Q(\iota(\phi^{(3)}_\xi (w)), \phi^{(3)}_\xi(w))}{Q(\iota(w),w)}
\end{equation*}
This does not depend on the choice of $w$, and since $Q(\iota(H^{(2)}),H^{(2)}) < 0$ this defines a nonnegative, equivariant quadratic form on the transverse distribution $T^t \Dom$. 

\begin{prop}
    $\G_1 = 7 \Phi^* h_\Dom$.
\end{prop}

\begin{proof}
    Let $\varphi$ be a unit volume torsion-free $G_2$-structure on $M$ and take $w = [\varphi] + [\Theta(\varphi)] \in H^{(3)}_\varphi$, so that
    \begin{equation*}
        Q(\iota(w),w) = Q([\varphi] - [\Theta(\varphi)],[\varphi] + [\Theta(\varphi)]) = 14.
    \end{equation*}
    Let $\eta \in \mathscr{H}^3_{27}(M,\varphi)$, identified with an element of $T_{\varphi \D} \M_1$, and let $(\phi^{(3)}_\eta,\phi^{(2)}_\eta,\phi^{(1)}_\eta)$ be the triple of linear maps representing $T \Phi(\eta) \in T_{\flag_\varphi} \Dom$. By Theorem \ref{thm:horizontal} we have
    \begin{equation*}
        Q(\iota(\phi^{(3)}_\eta(w)),\phi^{(3)}_\eta(w)) = Q([\eta] + [*_\varphi \eta], [\eta]-[*_\varphi \eta]) = - 2 \int |\eta|^2_\varphi \mu_\varphi .
    \end{equation*}
    Thus $\Phi^* h_\Dom(\eta,\eta) = \G_1(\eta,\eta)/7$. 
\end{proof}

In the same way, $\Phi$ determines the Yukawa coupling $\Xi$ on $\M$; by Lemma \ref{lem:propyukawa}, $\Xi = - dt \otimes \G + \Xi_1$ and thus we just need to show that $\Xi_1$ is the pull-back of an equivariant cubic form defined on the transverse distribution in $\Dom$. If $\flag \in \Dom$ and $\xi,\xi^\prime, \xi^{\prime\prime} \in T^t_\flag \Dom$, each transverse vector is represented by a triple of linear maps $(\phi^{(3)}_\xi,\phi^{(2)}_\xi,\phi^{(1)}_\xi)$ and similarly for $\xi^\prime$ and $\xi^{\prime\prime}$. Since each $\phi^{(p)}_\xi$ maps $H^{(p)}$ to $H^{(p-1)}$, the composition $\phi^{(1)}_\xi \circ \phi^{(2)}_{\xi^\prime} \circ \phi^{(3)}_{\xi^{\prime\prime}}$ defines a linear map from $H^{(3)}$ to $H^{(0)}$. Both are $1$-dimensional spaces, and thus there exists a unique $\Xi_\Dom(\xi,\xi^\prime,\xi^{\prime\prime})$ such that
\begin{equation*}
    \phi^{(1)}_\xi \circ \phi^{(2)}_{\xi^\prime} \circ \phi^{(3)}_{\xi^{\prime\prime}}(w) = - \Xi_\Dom(\xi,\xi^\prime,\xi^{\prime\prime}) \cdot \iota(w), ~~~ \forall w \in H^{(3)}. 
\end{equation*}
This defines equivariantly a cubic form $\Xi_\Dom$ on $T^t \Dom$.

\begin{prop}
    $\Xi_1 = 7 \Phi^* \Xi_\Dom$.
\end{prop}

\begin{proof}
    Let $\varphi$ be a unit-volume torsion-free $G_2$-structure on $M$ and $\eta,\eta^\prime,\eta^{\prime\prime} \in \mathscr{H}^3_{27}(M,\varphi)$. Theorem \ref{thm:horizontal} yields:
    \begin{align*}
        \phi^{(1)}_\eta \circ \phi^{(2)}_{\eta^\prime} \circ \phi^{(3)}_{\eta^{\prime\prime}}([\varphi]+[\Theta(\varphi)]) & = \frac{1}{7} \int \langle \pi_{27} \mathscr{H}(h^\prime \cdot \eta^{\prime\prime}) , \eta \rangle_\varphi \mu_\varphi \cdot ([\varphi]-[\Theta(\varphi)])  \\
            & = \frac{1}{7} \int \langle h^\prime \cdot \eta^{\prime\prime} , \eta \rangle_\varphi \mu_\varphi \cdot ([\varphi]-[\Theta(\varphi)]) 
    \end{align*}
    which proves the proposition.
\end{proof}

\begin{rem}
    This is similar to the way the Yukawa coupling is defined on the moduli spaces of Calabi--Yau threefolds, as described by Bryant and Griffiths \cite{bryant1983some}.
\end{rem}

\begin{rem}
    The way we defined it, $\Xi_\Dom$ is actually not a symmetric cubic form on $T^t \Dom$. However, we if consider the transversality condition as an exterior differential system on $\Dom$, one can prove that the restriction of $\Xi_\Dom$ to any integral element is fully symmetric. Hence $\Xi_\Dom$ will be symmetric along any integral submanifold of the transverse distribution.
\end{rem}

    \subsection{A condition for $\Phi$ to be totally geodesic}      \label{subsection:condition}

In this part, we relate the geometry of the immersion $\Phi : \M \rightarrow \Dom$ with the computations of the previous section and refine the observations of \S\ref{subsection:curvatures}. Our main result is that the covariant derivative of the Yukawa coupling $\Xi$, or equivalently the extra term $\mathscr{E}_{abcd} + \mathscr{E}_{cabd} + \mathscr{E}_{cbad}$, essentially characterises the second fundamental form of $\Phi(\M)$ inside the domain $\Dom$. More precisely, we have:

\begin{thm}     \label{thm:totgeod}
    The Yukawa coupling is a parallel tensor if and only if $\Phi : \M \rightarrow \Dom$ is a totally geodesic immersion. Moreover, if these conditions are satisfied then the Levi-Civita connections of $\G$ and $\Phi^* g_\Dom$ coincide and $(\M,\G)$ is a locally symmetric space with nonpositive sectional curvature.
\end{thm}

\begin{rem}
    This result is a $G_2$-counterpart for theorems of Liu--Yin \cite{liu2014quantum} and Wei \cite{wei2017some} for moduli spaces of Calabi--Yau $3$- and $4$-folds.
\end{rem}

For the proof of the theorem, first remark that since $\Phi$ is a horizontal immersion, $\Phi(\M)$ is totally geodesic in $\Dom$ if and only if the composition $q \circ \Phi : \M \rightarrow S^2_+(H^3)^*$ is a totally geodesic immersion. Moreover, the metrics $\Phi^* g_\Dom$ and $\Phi^* q^* g_{S^2_+}$ coincide, and therefore it is enough to prove that the results hold for the map $q \circ \Phi$ instead of $\Phi$. The advantage of working in $S^2_+(H^3)^*$ instead of $\Dom$ is that we can work directly in coordinates which are compatible with affine coordinates on $\M$.

First we need to introduce some notations. For the remainder of this part we will fix a basis $(u_0,\ldots,u_n)$ of $H^3$ and denote by $(x^0,\ldots,x^n)$ the associated system of coordinates, considered as local coordinates on $\M$. Any symmetric bilinear form $q \in S^2(H^3)^*$ can be written uniquely $q = q_{kl} dx^kdx^l$ where $q_{kl} = q_{lk}$ and as before we write $dx^k dx^l$ as a short-hand for the tensor product $dx^k \otimes dx^l$. Then $(q_{kl})_{1 \leq k \leq l \leq n}$ define global coordinates on the open cone $S^2_+ (H^3)^*$ of inner products on $H^3$. Let us write the canonical symmetric metric of $g_{S^2_+}$ in these coordinates. Let us pick $q \in S^2_+(H^3)^*$ and $\dot q \in S^2(H^3)^* \simeq T_q S^2_+(H^3)^*$ written as
\begin{equation*}
    q = q_{kl} dx^k dx^l \in S^2_+(H^3)^*, ~~~~ \dot q = \dot q_{kl} dx^k dx^l \in T_q S^2_+ (H^3)^*  .
\end{equation*}
There exists a unique $q$-self-adjoint endomorphism $a$ of $H^3$ such that $\dot q = \left. \frac{d}{dt} \right|_{t=0} (e^{ta})^*q = q(a \cdot , \cdot) + q(\cdot, a \cdot) = 2 q(a\cdot,\cdot)$ and by definition $|\dot q|_q^2 = \tr(a^2) = a^k_l a^l_k$. In coordinates we have $a^k_l = \frac{1}{2} q^{kr} \dot q_{rl}$ and hence it follows that
\begin{equation*}
    g_{S^2_+}(\dot q, \dot q) = \frac{1}{4} q^{kl}q^{rs} \dot q_{kr} \dot q_{sl} .
\end{equation*}
One can use this expression to compute the Christoffel symbols of $g_{S^2_+}$ and deduce that its Levi-Civita connection $\overline{\nabla}$ can be characterised as follows:

\begin{lem}     \label{lem:christsymbolsh}
    Let $\dot q = \dot q_{kl} dx^k dx^l$ and $\dot q^\prime = \dot q^\prime_{kl} dx^kdx^l$ be vector fields with constant coefficients on $S^2_+(H^3)^*$. Then the covariant derivative $\overline{\nabla}_{\dot q} \dot q^\prime$ is given by
    \begin{equation*}
        \overline{\nabla}_{\dot q} \dot q^\prime =  - \frac{1}{2} q^{rs} (\dot q_{kr} \dot q^\prime_{ls} + \dot q_{lr} \dot q^\prime_{ks}) dx^k dx^l .
    \end{equation*}
\end{lem}

If $\varphi$ is a torsion-free $G_2$-structure on $M$ then $q \circ \Phi(\varphi \D)$ is the $L^2$-inner product induced by $\varphi$ on $H^3 \simeq \mathscr{H}^3(M,g_\varphi)$, and therefore in the coordinates $x^a$ we have
\begin{equation}
    q(x) = e^{-\F/3} \G_{kl} dx^k dx^l 
\end{equation}
where the factor $e^{-\F/3} = \Vol$ compensates the volume normalisation in the definition of the metric $\G$. Thus as a subspace of $S^2_+ (H^3)^*$, $(q \circ \Phi)_* T\M$ is spanned by the vectors
\begin{equation}        \label{eq:spandelq}
    \frac{\partial q}{\partial x^a} = e^{-\F/3} \left( \F_{akl} - \frac{1}{3} \F_a \G_{kl} \right) dx^k dx^l , ~~~~ a = 0,\ldots,n.
\end{equation}
With a small abuse, we still denote by $\overline{\nabla}$ the pull-back connection $(q \circ \Phi)^* \overline{\nabla}$, considered as a connection on the trivial vector bundle $\M \times S^2(H^3)^*$. Using Lemma \ref{lem:christsymbolsh}, we have
\begin{align*}
    \overline{\nabla}_{\partial_a} \frac{\partial}{\partial x^b} = & e^{-\F/3} \left( \F_{abkl} - \frac{1}{3} \G_{ab}\G_{kl} - \frac{1}{3} \F_a\F_{bkl} - \frac{1}{3} \F_b \F_{akl} + \frac{1}{9} \F_a \F_b \G_{kl} 
    \right) dx^k dx^l \\
        & - \frac{1}{2} e^{-\F/3} \G^{rs} \left(\F_{akr} - \frac{1}{3}\F_a \G_{kr} \right) \left(\F_{bls} - \frac{1}{3} \F_b \G_{ls} \right) dx^k dx^l \\
        & - \frac{1}{2} e^{-\F/3} \G^{rs} \left(\F_{aks} - \frac{1}{3}\F_a \G_{ks} \right) \left(\F_{blr} - \frac{1}{3} \F_b \G_{lr} \right) dx^k dx^l \\
        = & e^{-\F/3} \left( \F_{abkl} - \frac{1}{2} \G^{rs} \F_{akr} \F_{bls} - \frac{1}{2} \G^{rs} \F_{aks} \F_{blr} - \frac{1}{3} \G_{ab}\G_{kl} \right) dx^k dx^l .
\end{align*}
In the next proposition, we rewrite this expression in a more intrinsic way:

\begin{prop}        \label{prop:connections}
    The connections $\overline{\nabla}$, $\nabla^\G$ and the covariant derivative of the Yukawa coupling are related by
    \begin{equation*}
        \overline{\nabla}_{\partial_a} \frac{\partial}{\partial x^b} = \nabla^\G_{\partial_a} \frac{\partial}{\partial x^b} + 2 e^{-\F/3} \nabla^\G_a \Xi_{bkl} dx^k dx^l
    \end{equation*}
    where we see $\nabla^\G_{\partial_a} \frac{\partial}{\partial x^b}$ as an element of $S^2(H^3)^*$ via the inclusion $(q \circ \Phi)_* T\M \subset S^2 (H^3)^*$.
\end{prop}

\begin{proof}
    By our previous computation we have
    \begin{align}       \label{eq:nablabarab}
        \begin{split}
            \overline{\nabla}_{\partial_a} \frac{\partial}{\partial x^b} = & e^{-\F/3}\left( \F_{abkl} - \frac{1}{2} \G^{rs}( \F_{akr} \F_{bls} + \F_{aks} \F_{blr} + \F_{abr}\F_{kls} ) \right) dx^k dx^l \\
            & + e^{-\F/3} \left( \frac{1}{2} \G^{rs} \F_{abr} \F_{kls} - \frac{1}{3} \G_{ab} \G_{kl} \right) dx^k dx^l .
        \end{split}
    \end{align}
    Comparing with the expression of the covariant derivative of the Yukawa coupling given in \S\ref{subsection:curvatures}, the term on the first line is $2 e^{-\F/3} \nabla^\G_a \Xi_{bkl}$. We need to rewrite the second term using the special properties of the function $\F$ and its derivatives. By Remark \ref{rem:idgf} and Lemma \ref{lem:propyukawa} we have the identities
    \begin{equation}        \label{eq:idxf}
        x^m \G_{sm} = - \F_s, ~~~ x^m \F_{mab} = - 2 \G_{ab} .
    \end{equation}
    Now let us compute:
    \begin{equation*}
        \frac{1}{2} \G^{rs} \F_{abr} \F_s = - \frac{1}{2} \G^{rs} \G_{sm} x^m \F_{abr} = - \frac{1}{2} x^r \F_{abr} = \G_{ab}
    \end{equation*}
    and thus the term on the second line of \eqref{eq:nablabarab} can be written as
    \begin{equation*}
        e^{-\F/3} \left( \frac{1}{2} \G^{rs} \F_{abr} \F_{kls} - \frac{1}{3} \G_{ab} \G_{kl} \right) = \frac{1}{2} \G^{rs} \F_{abr} \cdot e^{-\F/3} \left(\F_{kls} - \frac{1}{3} \F_s \G_{kl} \right) .
    \end{equation*}
    By \eqref{eq:gchristsymb}, $\frac{1}{2} \G^{rs} \F_{abs}$ are the Christoffel symbols of the metric $\G$ in the affine coordinates $x^k$, whilst $e^{-\F/3} (\F_{kls} - \frac{1}{3} \F_s \G_{kl})dx^kdx^l$ is just $\frac{\partial q}{\partial x^s}$ by \eqref{eq:spandelq}. Hence
    \begin{equation*}
        \frac{1}{2} \G^{rs} \F_{abr} \cdot e^{-\F/3} \left(\F_{kls} - \frac{1}{3} \F_s \G_{kl} \right) = \nabla^\G_{\partial_a} \frac{\partial}{\partial x^b}
    \end{equation*}
    which finishes the proof of the lemma.
\end{proof}

After these preliminary computations, we are now equipped to prove the theorem:

\begin{proof}[Proof of Theorem \ref{thm:totgeod}]
    From the previous proposition it follows that if $\nabla^\G \Xi \equiv 0$ then the connections $\overline{\nabla}$ and $\nabla^\G$ coincide. In that case, $\overline{\nabla}$ has no component along the normal space of $q \circ \Phi(\M)$ in $S^2_+ (H^3)^*$ and thus $q \circ \Phi$ is a totally geodesic immersion. This also implies that $\overline{\nabla}$ is equal to its projection on the tangent space of $\M$, which is exactly the Levi-Civita connection of $\Phi^* g_\Dom$, and therefore $\G$ and $\Phi^* g_\Dom$ have the same Levi-Civita connection. Moreover, since $S^2_+(H^3)^*$ is a symmetric space with nonpositive sectional curvature \cite{helgason1979differential}, the metric $\Phi^* g_\Dom = \Phi^* q^* g_{S^2_+}$ is locally symmetric and has nonpositive sectional curvature, and as these properties only depend on the Levi-Civita connection they must also be satisfied by the metric $\G$.

    It remains to prove that if $\Phi$ is a totally geodesic immersion then the Yukawa coupling is parallel. Thus let us assume that $\Phi$ is totally geodesic. Since the map $q : \Dom \rightarrow S^2_+(H^3)^*$ is a Riemannian fibration and $\Phi$ is a horizontal map, it follows that $q \circ \Phi$ is also a totally geodesic immersion. Therefore, the connection $\overline{\nabla}$ of the bundle $\M \times S^2(H^3)^*$ must preserve the tangent space $T\M$ (seen as a subbundle). Given Proposition \ref{prop:connections}, we deduce that for all $0 \leq a, b \leq n$, the quadratic form $\nabla^\G_a \Xi_{bkl}dx^kdx^l = \frac{1}{2} e^{\F /3} (\overline{\nabla}_{\partial_a} \frac{\partial}{\partial x^b} - \nabla^\G_{\partial_a} \frac{\partial}{\partial x^b})$ is a section of $T\M$ when $\Phi$ is a totally geodesic immersion. We shall now prove that there exists a subbundle $E$ of $\M \times S^2(H^3)^*$ such that $T\M \oplus E = \M \times S^2(H^3)^*$ and $\nabla^\G_a \Xi_{bkl}dx^kdx^l$ is a section of $E$ for any $0 \leq a, b \leq n$. Once we have shown this, then when $\Phi$ is totally geodesic each $\nabla^\G_a \Xi_{bkl}dx^kdx^l$ is a section of both $T\M$ and its complement $E$ and therefore it must vanish, whence $\nabla^\G_a \Xi_{bkl} = 0$ for all $0 \leq a,b,k,l \leq n$ and the theorem is proved.
    
    Using local affine coordinates, let us define $E$ by
    \begin{equation*}
        E_x = \{ q \in S^2(H^3)^*, ~ q(\cdot ,x) = 0 \} \subset S^2(H^3)^* . 
    \end{equation*}
    The subspace $E_x$ has codimension $n+1$ in $S^2(H^3)^*$, that is $\codim(E_x) = \dim(T_x \M)$.  By Lemma \ref{lem:propyukawa}, for any $0 \leq a,b \leq n$ the covariant derivative of the Yukawa coupling satisfies
    \begin{equation*}
        x^r \nabla^\G_a \Xi_{bkr} = 0
    \end{equation*}
    and therefore the quadratic form $\nabla^\G_a \Xi_{bkl} dx^k dx^l$ takes values in $E_x$. 

    In order to prove that $E_x$ is a complement of $T_x \M$ in $S^2(H^3)^*$, we need to prove that the $n+1$ linear forms $\frac{\partial q}{\partial x^a}(\cdot,x) \in (H^3)^*$ are linearly independent. By \eqref{eq:spandelq}, we have
    \begin{align*}
        e^{\F/3} \frac{\partial q}{\partial x^a}(\cdot,x) & = x^r\F_{akr} dx^k - \frac{1}{3} x^r \F_a \G_{kr}dx^k  \\
        & = - 2 \G_{ak}dx^k + \frac{1}{3} \F_a \F_k dx^k \\
        & = - 2\G \left( \frac{\partial}{\partial x^a}, \cdot \right) + \frac{1}{3} \frac{\partial \F}{\partial x^a} \cdot d\F .
    \end{align*}
    where we used the identities \eqref{eq:idxf} to pass from the first to the second line. After a linear change of coordinates, we may assume that $\frac{\partial}{\partial x^1}, \ldots, \frac{\partial}{\partial x^n}$ are tangent to the level set of $\F$ at the point $x$. Hence we just have $\frac{\partial q}{\partial x^a}(\cdot, x) = -2 e^{-\F/3} \G(\partial_a, \cdot)$ for $1 \leq a \leq n$, and this gives $n$ independent linear forms. Moreover, using Remark \ref{rem:idgf} we can compute that
    \begin{equation*}
        e^{\F/3} x^a \frac{\partial q}{\partial x^a}(\cdot, x) = - 2 x^a \G \left( \frac{\partial}{\partial x^a}, \cdot \right) + \frac{1}{3} x^a\frac{\partial \F}{\partial x^a} \cdot d\F = 2 d\F - \frac{7}{3} d\F = - \frac{1}{3} d\F = \frac{1}{3} \G(x,\cdot)
    \end{equation*}
    and this gives another linear form independent from the previous ones, since $x^a \frac{\partial}{\partial x^a}$ is linearly independent of $\frac{\partial}{\partial x^1}, \ldots , \frac{\partial}{\partial x^r}$ (as $x^a \F_a \neq 0$). Thus the $n+1$ linear forms $\frac{\partial q}{\partial x^a}(\cdot, x)$, $a = 0,\ldots, n$, are independent. Hence $T_x\M$ is a complement of $E_x$ in $S^2(H^3)^*$, which completes the proof of the theorem.
\end{proof}

    \subsection{Further comments}       \label{subsection:comments}

We finish this article with some comments concerning the relation between the map $\Phi$ and the map $\varphi \D \in \M \mapsto [\varphi]+[\Theta(\varphi)] \in H$, which was proved by Joyce to be a Lagrangian immersion \cite{joyce1996compactii}. The relation is better explained using the theory of exterior differential systems, and is very similar to the results of Bryant--Griffiths \cite{bryant1983some} about the periods of Calabi--Yau threefolds. 

The idea is the following. Note that $[\varphi]+[\Theta(\varphi)]$ spans $H^{(3)}_\varphi$, and if we restrict our attention to $\M_1$ it follows that the map $\varphi \D \in \M_1 \mapsto H^{(3)}_\varphi \in \mathbb{P}(H)$ is a Legendrian immersion for the contact system induced by $Q$ on $\mathbb{P}(H)$. This map takes values in the open subset $U \subset \mathbb{P}(H)$ defined by
\begin{equation*}
    U = \{ \Span \{w \}, ~ Q(\iota(w),w) > 0 \} .
\end{equation*}
There is a homogeneous fibration $\Dom \rightarrow U$ mapping $\flag \in \Dom$ to $H^{(3)}$, and it turns out that this identifies $\Dom$ with an open subset of the space of \emph{maximal integral elements} of the contact system over $U$. Moreover, the exterior differential system on $\Dom$ corresponding to the transversality condition can be shown to be the restriction to $\Dom$ of the \emph{first prolongation} of the contact system. Below we sketch the proofs of the above claims.

Let us consider coordinates $(w^0,\ldots,w^n,w_0,\ldots,w_n)$ on $H$ such that
\begin{equation*}
    Q = \sum_{j=0}^n dw_j \wedge dw^j, ~~~ \text{and} ~~~ \iota(w_j,w^j) = (w^j,w_j) .
\end{equation*}
In homogeneous coordinates $[w^0 = 1:w^1:\cdots:w^n,w_0:\cdots:w_n]$ on $\mathbb{P}(H)$, the contact system is generated by the $1$-form
\begin{equation}        \label{eq:gammacanonical}
    \alpha =  dw_0 + \sum_{j=1}^n w^j dw_j - w_j dw^j
\end{equation}
and in particular:
\begin{equation}        \label{eq:dgammacanonical}
    d\alpha = -2 \sum_{j=1}^n dw_j \wedge dw^j .
\end{equation}
It is a classical fact that the integral submanifolds of the contact system have dimension at most $n$. Let us define $V_n(U)$ to be the space of $n$-dimensional integral elements of the contact system restricted to $U$. If $H^{(3)} \in U$ and we choose our previous coordinates such that $H^{(3)} = \{ dw^1 = \cdots = dw^n = dw_0 = \cdots dw_n = 0 \}$, then the integral elements of the contact system lying over $H^{(3)}$ are defined by the equations
\begin{equation*}
    dw_0 = 0, ~~ \text{and} ~~ \sum_{j=1}^n dw_j \wedge dw^j = 0 .
\end{equation*}
Therefore the space of such integral elements can be defined as the set of $n$-dimensional subspaces $H^{(2)} \subset H$ satisfying:
\begin{equation*}
    Q(\iota(H^{(3)}),H^{(2)}) = Q(H^{(3)},H^{(2)}) = Q(H^{(2)},H^{(2)}) = 0.
\end{equation*}
Thus if an integral element satisfies the additional open property that $Q(\iota(H^{(2)}),H^{(2)}) < 0$, then comparing with properties \textbf{(1)}, \textbf{(2)} and \textbf{(3)} we see that $H^{(3)}$, $H^{(2)}$, $H^{(1)} = \iota(H^{(2)})$ and $H^{(0)} = \iota(H^{(3)})$ define an element of $\Dom$. Conversely, for any $\flag \in \Dom$, $H^{(2)}$ is an integral element of the contact system lying over $H^{(3)} \in U$. This proves that $\Dom$ can be identified with an open subset of the space of maximal integral elements of the contact system over $U$, and the fact that the restriction to $\Dom$ of the first prolongation of the contact system coincides with the exterior differential system corresponding to the transversality condition is essentially a consequence of Lemma \ref{lem:hortangent}.

Therefore, there is a one-to-one correspondence between the Legendrian submanifolds $N \subset U \subset \mathbb{P}(H)$ whose tangent space satisfies $Q(\iota(TN),TN) < 0$ and the maximal transverse submanifolds of $\Dom$ such that the restriction of the quadratic form $h_\Dom$ is nondegenerate. Under this correspondence, the immersion $\M_1 \rightarrow \mathbb{P}(H)$ is associated with the map $\Phi : \M_1 \rightarrow \Dom$, and up to a numerical factor the metric $\G_1$ is precisely the restriction of the quadratic form $h_\Dom$.


\addcontentsline{toc}{section}{References}

\small

\bibliographystyle{abbrv}
\bibliography{ref_moduligeom}

\begin{thebibliography}{10}

\bibitem{beasley2002note}
C.~Beasley and E.~Witten.
\newblock A note on fluxes and superpotentials in m-theory compactifications on manifolds of {$G_2$} holonomy.
\newblock {\em J. High Energy Phys.}, 2002(07):046, 2002.
\newblock \href{https://doi.org/10.1088/1126-6708/2002/07/046}{https://doi.org/10.1088/1126-6708/2002/07/046}.

\bibitem{bonan1966sur}
E.~Bonan.
\newblock Sur les vari\'{e}t\'{e}s riemanniennes \`{a} groupe d'holonomie {$G_2$} ou {$Spin(7)$}.
\newblock In {\em C. R. Acad. Sc. Paris}, volume 262, 1966.
\newblock \href{https://gallica.bnf.fr/ark:/12148/bpt6k6236863n/f141.item}{https://gallica.bnf.fr/ark:/12148/bpt6k6236863n/f141.item}.

\bibitem{bryant1987metrics}
R.~L. Bryant.
\newblock Metrics with exceptional holonomy.
\newblock {\em Ann. Math.}, pages 525--576, 1987.
\newblock \href{https://doi.org/10.2307/1971360}{https://doi.org/10.2307/1971360}.

\bibitem{bryant1983some}
R.~L. Bryant and P.~A. Griffiths.
\newblock Some observations on the infinitesimal period relations for regular threefolds with trivial canonical bundle.
\newblock {\em Arithmetic and Geometry: Papers Dedicated to IR Shafarevich on the Occasion of His Sixtieth Birthday. Volume II: Geometry}, pages 77--102, 1983.

\bibitem{bryant1989construction}
R.~L. Bryant and S.~M. Salamon.
\newblock On the construction of some complete metrics with exceptional holonomy.
\newblock {\em Duke Math. J.}, 58(3):829--850, 1989.
\newblock \href{https://doi.org/10.1215/S0012-7094-89-05839-0}{https://doi.org/10.1215/S0012-7094-89-05839-0}.

\bibitem{cheeger1971splitting}
J.~Cheeger and D.~Gromoll.
\newblock The splitting theorem for manifolds of nonnegative {R}icci curvature.
\newblock {\em J. Differ. Geom.}, 6(1):119--128, 1971.
\newblock \href{https://doi.org/10.4310/jdg/1214430220}{https://doi.org/10.4310/jdg/1214430220}.

\bibitem{crowley2025path}
D.~Crowley, S.~Goette, and T.~Hertl.
\newblock Path components of {$G_2$}-moduli spaces may be non-aspherical.
\newblock {\em \href{https://arxiv.org/abs/2503.15829}{arXiv:2503.15829}}, 2025.

\bibitem{crowley2015analytic}
D.~Crowley, S.~Goette, and J.~Nordstr\"{o}m.
\newblock An analytic invariant of {$G_2$}-manifolds.
\newblock {\em Invent. Math.}, 239(3):865--907, 2025.
\newblock \href{https://doi.org/10.1007/s00222-024-01310-z}{https://doi.org/10.1007/s00222-024-01310-z}.

\bibitem{fernandez1982riemannian}
M.~Fern{\'a}ndez and A.~Gray.
\newblock Riemannian manifolds with structure group {$G_2$}.
\newblock {\em Ann. Mat. Pura Appl.}, 132(1):19--45, 1982.
\newblock \href{https://doi.org/10.1007/BF01760975}{https://doi.org/10.1007/BF01760975}.

\bibitem{griffiths1968periodsi}
P.~A. Griffiths.
\newblock Periods of integrals on algebraic manifolds, {I}.
\newblock {\em Am. J. Math.}, 90(2):568--626, 1968.
\newblock \href{https://doi.org/10.2307/2373545}{https://doi.org/10.2307/2373545}.

\bibitem{griffiths1968periodsii}
P.~A. Griffiths.
\newblock Periods of integrals on algebraic manifolds, {II}.
\newblock {\em Am. J. Math.}, 90(3):805--865, 1968.
\newblock \href{https://doi.org/10.2307/2373485}{https://doi.org/10.2307/2373485}.

\bibitem{grigorian2010moduli}
S.~Grigorian.
\newblock Moduli spaces of {$G_2$} manifolds.
\newblock {\em Rev. Math. Phys.}, 22(09):1061--1097, 2010.
\newblock \href{https://doi.org/10.1142/S0129055X10004132}{https://doi.org/10.1142/S0129055X10004132}.

\bibitem{grigorian2009local}
S.~Grigorian and S.-T. Yau.
\newblock Local geometry of the {$G_2$} moduli space.
\newblock {\em Commun. Math. Phys.}, 287(2):459--488, 2009.
\newblock \href{https://doi.org/10.1007/s00220-008-0595-1}{https://doi.org/10.1007/s00220-008-0595-1}.

\bibitem{gukov2000solitons}
S.~Gukov.
\newblock Solitons, superpotentials and calibrations.
\newblock {\em Nucl. Phys. B}, 574(1-2):169--188, 2000.
\newblock \href{https://doi.org/10.1016/S0550-3213(00)00053-5}{https://doi.org/10.1016/S0550-3213(00)00053-5}.

\bibitem{gutowski2001moduli}
J.~Gutowski and G.~Papadopoulos.
\newblock Moduli spaces and brane solitons for {M}-theory compactifications on holonomy {$G_2$} manifolds.
\newblock {\em Nucl. Phys. B}, 615(1-3):237--265, 2001.
\newblock \href{https://doi.org/10.1016/S0550-3213(01)00419-9}{https://doi.org/10.1016/S0550-3213(01)00419-9}.

\bibitem{hatcher1978concordance}
A.~E. Hatcher.
\newblock Concordance spaces, higher simple-homotopy theory, and applications.
\newblock In {\em Algebraic and geometric topology (Proc. Sympos. Pure Math., Stanford Univ., Stanford, Calif., 1976), Part}, volume~1, pages 3--21, 1978.

\bibitem{helgason1979differential}
S.~Helgason.
\newblock {\em Differential geometry, Lie groups, and symmetric spaces}, volume~80.
\newblock New York, Academic Press, 1978.

\bibitem{hitchin2000geometry}
N.~Hitchin.
\newblock The geometry of threeforms in {$6$} and {$7$} dimensions.
\newblock {\em arXiv preprint \href{https://arxiv.org/abs/math/0010054}{arXiv:0010054}}, 2000.

\bibitem{house2005m}
T.~House and A.~Micu.
\newblock M-theory compactifications on manifolds with {$G_2$} structure.
\newblock {\em Class. Quantum Grav.}, 22(9):1709, 2005.
\newblock \href{https://doi.org/10.1088/0264-9381/22/9/016}{https://doi.org/10.1088/0264-9381/22/9/016}.

\bibitem{hsiang1976parametrized}
W.~C. Hsiang and R.~Sharpe.
\newblock Parametrized surgery and isotopy.
\newblock {\em Pac. J. Math.}, 67(2):401--459, 1976.
\newblock \href{http://doi.org/10.2140/pjm.1976.67.401}{http://doi.org/10.2140/pjm.1976.67.401}.

\bibitem{huybrechts2001products}
D.~Huybrechts.
\newblock Products of harmonic forms and rational curves.
\newblock {\em Doc. Math.}, 6:227--239, 2001.
\newblock \href{https://doi.org/10.4171/DM/102}{https://doi.org/10.4171/DM/102}.

\bibitem{joyce1996compacti}
D.~D. Joyce.
\newblock Compact {R}iemannian 7-manifolds with holonomy {$G_2$}. {I}.
\newblock {\em J. Differ. Geom.}, 43:291--328, 1996.
\newblock \href{https://doi.org/10.4310/jdg/1214458109}{https://doi.org/10.4310/jdg/1214458109}.

\bibitem{joyce1996compactii}
D.~D. Joyce.
\newblock Compact {R}iemannian 7-manifolds with holonomy {$G_2$}. {II}.
\newblock {\em J. Differ. Geom.}, 43:329--375, 1996.
\newblock \href{https://doi.org/10.4310/jdg/1214458110}{https://doi.org/10.4310/jdg/1214458110}.

\bibitem{joyce2000compact}
D.~D. Joyce.
\newblock {\em Compact manifolds with special holonomy}.
\newblock Oxford University Press, 2000.

\bibitem{karigiannis2009hodge}
S.~Karigiannis and N.~C. Leung.
\newblock Hodge theory for {$G_2$}-manifolds: {I}ntermediate {J}acobians and {A}bel--{J}acobi maps.
\newblock {\em Proc. London Math. Soc.}, 99(2):297--325, 2009.
\newblock \href{https://doi.org/10.1112/plms/pdp004}{https://doi.org/10.1112/plms/pdp004}.

\bibitem{karigiannis2025octonionic}
S.~Karigiannis and J.~Loftin.
\newblock Octonionic-algebraic structure and curvature of the {T}eichm{\"u}ller space of {$G_2$} manifolds.
\newblock {\em In preparation}.

\bibitem{langlais2024incompleteness}
T.~Langlais.
\newblock On the incompleteness of {$G_2$}-moduli spaces along degenerating families of {$G_2$}-manifolds.
\newblock {\em Int. Math. Res. Not.}, 2025(6), 2025.
\newblock \href{https://doi.org/10.1093/imrn/rnaf069}{https://doi.org/10.1093/imrn/rnaf069}.

\bibitem{langlais2025phd}
T.~Langlais.
\newblock {\em {$G_2$} moduli spaces: geometry, periods and degenerate limits}.
\newblock PhD thesis, University of Oxford, 2025 (in preparation).

\bibitem{lee2009geometric}
J.-H. Lee and N.~C. Leung.
\newblock Geometric structures on {$G_2$} and {$Spin(7)$}-manifolds.
\newblock {\em Adv. Theor. Math. Phys.}, 13(1):1--31, 2009.

\bibitem{liu2014quantum}
K.~Liu and C.~Yin.
\newblock Quantum correction and the moduli spaces of {C}alabi--{Y}au manifolds.
\newblock {\em arXiv preprint \href{https://arxiv.org/abs/1411.0069}{arXiv:1411.0069}}, 2014.

\bibitem{lu2001hodge}
Z.~Lu.
\newblock On the {H}odge metric of the universal deformation space of {C}alabi--{Y}au threefolds.
\newblock {\em J. Geom. Anal.}, 11:103--118, 2001.
\newblock \href{https://doi.org/10.1007/BF02921956}{https://doi.org/10.1007/BF02921956}.

\bibitem{lu2004weil}
Z.~Lu and X.~Sun.
\newblock Weil--{P}etersson geometry on moduli space of polarized {C}alabi--{Y}au manifolds.
\newblock {\em J. Inst. Math. Jussieu}, 3(2):185--229, 2004.
\newblock \href{https://doi.org/10.1017/S1474748004000076}{https://doi.org/10.1017/S1474748004000076}.

\bibitem{lu2006weil}
Z.~Lu and X.~Sun.
\newblock On the {W}eil--{P}etersson volume and the first {C}hern class of the moduli space of {C}alabi—{Y}au manifolds.
\newblock {\em Commun. Math. Phys.}, 261(2):297--322, 2006.
\newblock \href{https://doi.org/10.1007/s00220-005-1441-3}{https://doi.org/10.1007/s00220-005-1441-3}.

\bibitem{schmid1973variation}
W.~Schmid.
\newblock Variation of {H}odge structure: the singularities of the period mapping.
\newblock {\em Invent. Math.}, 22(3):211--319, 1973.
\newblock \href{https://doi.org/10.1007/BF01389674}{https://doi.org/10.1007/BF01389674}.

\bibitem{shima2007geometry}
H.~Shima.
\newblock {\em The geometry of {H}essian structures}.
\newblock World Scientific Publishing, 2007.

\bibitem{tian1987smoothness}
G.~Tian.
\newblock Smoothness of the universal deformation space of compact {C}alabi--{Y}au manifolds and its {P}etersson--{W}eil metric.
\newblock In {\em Mathematical aspects of string theory}, pages 629--646. World Scientific, 1987.

\bibitem{todorov1989weil}
A.~N. Todorov.
\newblock The {W}eil--{P}etersson geometry of the moduli space of {$SU (n \geq 3)$} {C}alabi--{Y}au manifolds, {I}.
\newblock {\em Commun. Math. Phys.}, 126:325--346, 1989.
\newblock \href{https://doi.org/10.1007/BF02125128}{https://doi.org/10.1007/BF02125128}.

\bibitem{trenner2011asymptotic}
T.~Trenner and P.~M.~H. Wilson.
\newblock Asymptotic curvature of moduli spaces for {C}alabi--{Y}au threefolds.
\newblock {\em J. Geom. Anal.}, 21:409--428, 2011.
\newblock \href{https://doi.org/10.1007/s12220-010-9152-1}{https://doi.org/10.1007/s12220-010-9152-1}.

\bibitem{viehweg1991quasi}
E.~Viehweg.
\newblock {\em Quasi-projective moduli for polarized manifolds}, volume~30.
\newblock Springer, 1991.

\bibitem{wei2017some}
K.~Wei.
\newblock Some results of quantum correction on compact {C}alabi--{Y}au manifolds of dimension {$3$} and {$4$}.
\newblock {\em Differ. Geom. Appl.}, 52:94--120, 2017.
\newblock \href{https://doi.org/10.1016/j.difgeo.2017.03.008}{https://doi.org/10.1016/j.difgeo.2017.03.008}.

\bibitem{wilson2004sectional}
P.~M.~H. Wilson.
\newblock Sectional curvatures of {K}{\"a}hler moduli.
\newblock {\em Math. Ann.}, 330:631--664, 2004.
\newblock \href{https://doi.org/10.1007/s00208-004-0563-9}{https://doi.org/10.1007/s00208-004-0563-9}.

\bibitem{yau1978ricci}
S.-T. Yau.
\newblock On the {R}icci curvature of a compact {K}{\"a}hler manifold and the complex {M}onge--{A}mp{\`e}re equation, {I}.
\newblock {\em Commun. Pure Appl. Math.}, 31(3):339--411, 1978.
\newblock \href{https://doi.org/10.1002/cpa.3160310304}{https://doi.org/10.1002/cpa.3160310304}.

\end{thebibliography}

\end{document}